\documentclass[12pt,reqno]{amsart}

\usepackage[T1]{fontenc}
\usepackage[latin1]{inputenc}
\usepackage[english]{babel}
\usepackage{color}
\usepackage{tikz}
\usetikzlibrary{cd}
\tikzset{%
    symbol/.style={%
        draw=none,
        every to/.append style={%
            edge node={node [sloped, allow upside down, auto=false]{$#1$}}}
    }
}

\newcommand{\pb}{\ar[dr,phantom,at start,"{\scalebox{1.7}{$\lrcorner$}}"]}

\usepackage[scale=0.7]{geometry}

\usepackage{hyperref}
\usepackage{amsmath,amssymb}
\usepackage{amsthm}
\usepackage{mathtools}
\usepackage{thmtools}

\newtheorem{thm}{Theorem}
\newtheorem{lemma}[thm]{Lemma}
\newtheorem{prop}[thm]{Proposition}
\newtheorem{coro}[thm]{Corollary}
\theoremstyle{definition}
\newtheorem{defn}{Definition}
\newtheorem*{conv}{Convention}
\newtheorem{counter}{Counterexample}
\newtheorem{example}[counter]{Example}
\newtheorem*{remark}{Remark}

\newcommand{\C}{\mathcal{C}}

\newcommand{\F}{\mathcal{F}}
\newcommand{\E}{\mathcal{E}}

\newcommand{\M}{\mathcal{M}}

\newcommand{\W}{\mathbb{W}}
\newcommand{\X}{\mathbb{X}}
\newcommand{\Y}{\mathbb{Y}}
\newcommand{\Z}{\mathbb{Z}}
\newcommand{\N}{\mathbb{N}}
\newcommand{\Q}{\mathbb{Q}}
\newcommand{\bP}{\mathbb{P}}
\renewcommand{\phi}{\varphi}

\newcommand{\RG}{\mathbf{RG}}
\newcommand{\MRG}{\mathbf{MRG}}
\newcommand{\Grpd}{\mathbf{Grpd}}
\newcommand{\Cat}{\mathbf{Cat}}

\newcommand{\Eq}{\mathbf{Eq}}

\newcommand{\Simp}{\mathbf{Simp}}

\usepackage{enumitem}

\newlist{thmlist}{enumerate}{1}
\setlist[thmlist]{label=(\arabic{thmlisti}), ref=\thethm.(\arabic{thmlisti}),noitemsep}

\title[Fundamental groupoids for simplicial objects in Mal'tsev categories]{Fundamental groupoids for simplicial objects in Mal'tsev categories}
\author{Arnaud Duvieusart}
\address{Institut de Recherche en Mathématique et Physique\\
	Université catholique de Louvain\\
	Chemin du Cyclotron 2, 1348 Louvain-la-Neuve, Belgium}
\address{Institute of Mathematics\\	
	Czech Academy of Sciences\\
	\v{Z}itná 25, 115 67 Prague, Czech Republic}

\email{duvieusart@math.cas.cz}

\date{\today}

\begin{document}

\begin{abstract}
We show that the category of internal groupoids in an exact Mal'tsev category is reflective, and, moreover, a Birkhoff subcategory of the category of simplicial objects. We then characterize the central extensions of the corresponding Galois structure, and show that regular epimorphisms admit a relative monotone-light factorization system in the sense of Chikhladze. We also draw some comparison with Kan complexes. By comparing the reflections of simplicial objects and reflexive graphs into groupoids, we exhibit a connection with weighted commutators (as defined by Gran, Janelidze and Ursini).
\end{abstract}

\maketitle
\section*{Introduction}

Categorical Galois theory, as developed by G. Janelidze (\cite{J90,JK97,BJ01,J04}), is a general framework that allows the study of central extensions or coverings of the objects of a category. A large collection of examples has been given, ranging from the Galois theory of commutative rings of Magid (\cite{M74,CJM96}) and the theory of coverings of locally connected spaces to the central extensions of groups, Lie algebras, or, more generally, central extensions in exact Mal'tsev categories \cite{JK94}.

The main ingredient of this theory is the notion of \emph{Galois structure}, which is defined as an adjunction, with the right adjoint often taken to be fully faithful, and a class of morphisms in the codomain of the right adjoint, satisfying suitable conditions, in particular \emph{admissibility}, which amounts to the preservation by the left adjoint of certain pullbacks. For example, the inclusion of any Birkhoff subcategory of an exact Mal'tsev together with the class of regular epimorphisms always forms an admissible Galois structure (\cite{JK94}).

In \cite{BJ99}, Brown and Janelidze used this theory to describe what they called \emph{second order coverings} for simplicial sets, using the adjunction given by the nerve functor and the fundamental groupoid, and the class of Kan fibrations. In fact, they restriced their analysis to Kan complexes, as this condition implies the \emph{admissibility} of these objects for the corresponding Galois structure. Later Chikhladze introduced relative factorization systems, and showed that the induced relative factorization system for Kan fibrations is locally stable, so that the Galois structures induces a relative monotone-light factorization (\cite{C04}).

On the other hand, regular Mal'tsev categories were characterized in \cite{CKP93} as the categories in which the Kan condition holds for \emph{every} simplicial object, thus extending a theorem of Moore stating that the underlying simplicial set of a simplicial group is always a Kan complex. Moreover, regular epimorphisms in the category of simplicial objects then coincide with Kan fibrations. This suggests that the inclusion of groupoids into simplicial objects in any exact Mal'tsev category might induce an admissible Galois structure.

The main objective of this paper is to show that this is indeed the case, and, more precisely, that the category of groupoids in an exact Mal'tsev category is always a Birkhoff subcategory of the category of simplicial objects. The paper is organised as follows : we begin with some preliminaries, to fix notation and provide the background notions. We then construct the reflection of the category of simplicial objects into the subcategory of internal groupoids. Next, we characterize the central extensions for the induced Galois structure. In the next section we compare our construction with the homotopy relations for the simplices in a Kan complex, which are used to define its homotopy groupoid. Then we prove that the Galois structure admits a relative monotone-light factorization system. We end the paper with a discussion of reflexive graphs, seen as truncated simplicial objects.

\section{Preliminaries}

\subsection{Simplicial objects}

Let $\Delta$ denote the category of finite nonzero ordinals, with monotone functions as morphisms. For a given category $\C$, the category $\Simp(\C)$ of simplicial objects in $\C$ is the category of functors $\Delta^{op}\to \C$. Equivalently, an object $\X$ of $\Simp(\C)$ is a collection of objects $(X_n)_{n\in \mathbb{N}}$ together with face morphisms $d_i\colon X_n\to X_{n-1}$ for all $n>0$ and $0\leq i\leq n$, and degeneracy morphisms $s_i\colon X_n\to X_{n+1}$ for $n\geq 0$ and $0\leq i\leq n$, satisfying the following simplicial identities, whenever they make sense:
\begin{equation}\begin{cases}
d_id_j = d_{j-1}d_i & 0\leq i<j\leq n+1 \\
s_is_j = s_{j+1}s_i & 0\leq i\leq j \leq n \\
d_is_j = s_{j-1}d_i & 0\leq i<j\leq n-1 \\
d_is_j = 1 & i\in \{j,j+1\}\\
d_is_j = s_jd_{i-1} & 0\leq j<i-1\leq n-1.
\end{cases}\end{equation}
When necessary, we will write $d_i^\X$ or $s_i^\X$ to distinguish the face or degeneracy morphisms of different simplicial objects. A morphism $f\colon \X\to \Y$ in $\Simp(\C)$ is then a collection of morphisms $f_n\colon X_n\to Y_n$ that commute with face and degeneracy morphisms, in the sense that $d_i^\Y f_{n+1}=f_{n}d_i^\X$ and $s_i^\Y f_{n}=f_{n+1}s_i^\X$ for all $i,n$.

If $\X$ is a simplicial object, we will denote $Dec(\X)$ the \emph{décalage} of $\X$ \cite{I72}, which is the simplicial object $(X_{n+1})_{n\geq 0}$, whose face and degeneracies are the same as those of $\X$, without the last faces $d_{n+1}\colon X_{n+1}\to X_n$ and last degeneracies $s_{n}\colon X_{n}\to X_{n+1}$ for all $n\geq 1$. The simplicial identities imply that the morphisms $d_{n+1}\colon X_{n+1}\to X_n$ form a morphism of simplicial objects $\epsilon_\X\colon Dec(\X)\to \X$. Since all these morphisms are split (and thus regular) epimorphisms, $\epsilon$ is a regular epimorphism in $\Simp(\C)$, although it need not be a split epimorphism. Notice that $Dec$ defines an endofunctor of $\Simp(\C)$, and $\epsilon$ is a natural transformation from $Dec$ to the identity endofunctor.

$\Delta$ is a skeleton of the category of non-empty finite totally ordered sets, and there exists exactly one functor from the latter category to $\Delta$. In particular, since the poset $\mathcal{P}_{f,n.e.}(\N)$ of non-empty finite subsets of $\N$ (ordered by inclusion) can be seen as a subcategory of non-empty finite totally ordered sets, there is a unique functor $\Phi\colon \mathcal{P}_{f,n.e.}(\N)\to \Delta$ that maps any set with $n+1$ elements to $\{0,\dots,n\}$ and any inclusion map to an injective morphism in $\Delta$.

For a given simplicial object $\X$, and for every $n\geq 2$, one can consider the restriction of $\Phi$ to the poset of proper subsets of $\{0,1,\dots, n\}$; taking the opposite functor and composing with $\X\colon \Delta^{op}\to \C$ gives a diagram in $\C$. The limit of this diagram is the $n$-th \emph{simplicial kernel} of $\X$, and denoted $K_n(\X)$. In particular, we have morphisms $\mu_i\colon K_n(\X)\to X_{n-1}$ for $i=0,\dots,n$, satisfying $d_i\mu_j=d_{j-1}\mu_i$ for all $0\leq i<j\leq n$, and the morphisms $\mu_i$ are universal with this property. Thus the face morphisms $d_0,\dots, d_n\colon X_n\to X_{n-1}$ induce a canonical morphism $\kappa_n\colon X_n\to K_n(\X)$. Following \cite{EGoV12}, we say that $\X$ is exact at $X_{n-1}$ if $\kappa_n$ is a regular epimorphism, and exact if it is exact at $X_n$ for all $n\geq 1$.

Moreover, for every $n\geq 2$ and $0\leq k\leq n$, we can also restrict $\Phi$ to the poset of proper subsets of $\{0,\dots, n\}$ that contain $k$, and then compose the opposite functor with $\X$. The limit of this diagram is the object of $(n,k)$-horns $\Lambda^n_k(\X)$, and it is equipped with morphisms $\nu_i\colon \Lambda^n_k(\X)\to X_{n-1}$ for $0\leq i\leq n$ and $i\neq k$ that satisfy the identities $d_i\nu_j=d_{j-1}\nu_i$ for all $0\leq i<j\leq n$ and $i\neq k\neq j$, and are universal with this property. There is then also a canonical arrow $\lambda_k^n\colon X_n\to \Lambda_k^n(\X)$ induced by the face morphisms $d_i\colon X_n\to X_{n-1}$ for $i\neq k$, and $\X$ is said to satisfy the Kan property if all these morphisms are regular epimorphisms. Moreover, a morphism $f\colon \X\to \Y$ between simplicial objects is called a \emph{Kan fibration} if for all $n$ and $k$ the canonical arrow $\theta_{k}^n$ in the diagram

\begin{equation}
\begin{tikzcd} {X_n} \ar[drr, bend left=15,"{f_n}"]\ar[ddr,bend right=15,"{\lambda_k^n}"'] \ar[dr,"{\theta_k^n}" description] & & \\ & {\Lambda_k^n(\X)\times_{\Lambda_k^n(\Y)} Y } \pb \ar[r] \ar[d] & {Y_n} \ar[d,"{\lambda_k^n}"]\\ & {\Lambda_k^n(\X)} \ar[r,"{\Lambda_k^n(f)}"'] & \Lambda_k^n(\Y) \end{tikzcd}\label{eq:kan_diag}
\end{equation}
(where the inner square is a pullback) is a regular epimorphism.

For every $n\geq 1$, we denote $\Delta_n$ the full subcategory of $\Delta$ consisting of the ordinals with $n+1$ elements or less, and $\Simp_n(\C)$ the category of functors $\Delta_n^{op}\to \C$, whose objects we called $n$-truncated simplicial objects. The inclusion $\Delta_n \hookrightarrow \Delta$ then induces by precomposition the truncation functor $\Simp(\C)\to \Simp_n(\C)$.

An internal reflexive graph in $\C$ is simply a $1$-truncated simplicial object. A \emph{multiplicative graph} is then a reflexive graph endowed with a partial multiplication $m\colon X_1\times_{X_0}X_1\to X_1$ that is unital and compatible with the domain and codomain morphisms (\cite{CPP92}), and an internal category is a multiplicative graph whose multiplication is associative. All these conditions can be summarized by saying that an internal multiplicative graph is an object of $\Simp_2(\C)$, such that the square
\[\begin{tikzcd}X_{2}\ar[r,"d_2"]\ar[d,"d_0"']& X_{1}\ar[d,"d_0"]\\ X_{1} \ar[r,"d_{1}"']& X_{0}
\end{tikzcd}\]
is a pullback, and an internal category is an object of $\Simp_3(\C)$ such that the square above and the square
\[\begin{tikzcd}X_{3}\ar[r,"d_3"]\ar[d,"d_0"']& X_{2}\ar[d,"d_0"]\\ X_{2} \ar[r,"d_{2}"']& X_{1}
\end{tikzcd}\]
are pullbacks. Moreover, an internal category is an internal groupoid if and only if any of the squares
\[\begin{tikzcd}X_{2}\ar[r,"d_1"]\ar[d,"d_0"']& X_{1}\ar[d,"d_0"]\\ X_{1} \ar[r,"d_{0}"']& X_{0}
\end{tikzcd}\qquad \text{and}\qquad \begin{tikzcd}X_{2}\ar[r,"d_2"]\ar[d,"d_1"']& X_{1}\ar[d,"d_1"]\\ X_{1} \ar[r,"d_{1}"']& X_{0}
\end{tikzcd}\]
is a pullback. Internal functors are also the same thing as (restricted) simplicial morphisms. Moreover, any internal category can be extended to a simplicial object by simply taking its nerve. From now on we will thus consider $\Cat(\C)$ and $\Grpd(\C)$ as full subcategories of $\Simp(\C)$; more precisely, a simplicial object $\X$ is an internal category if and only if the commutative square
\[\begin{tikzcd} X_{n}\ar[r,"d_0"]\ar[d,"d_{n}"']& X_{n-1}\ar[d,"d_{n-1}"]\\ X_{n-1}\ar[r,"d_{0}"']& X_{n-2}\end{tikzcd}\]
is a pullback for all $n\geq 2$.

\subsection{Mal'tsev categories and higher extensions}

A finitely complete category $\C$ is called a \emph{Mal'tsev category} if every internal reflexive relation in $\C$ is an equivalence relation \cite{CKP93,CLP91,CPP92,BB04}; when $\C$ is a regular category, this condition holds if and only if the composition $R\circ S$ of two equivalence relations $R,S$ on the same object $X$ is an equivalence relation. When this is the case, $R\circ S$ is in fact the join of $R$ and $S$ in the poset of equivalence relations of $X$. Accordingly this poset is a lattice. In fact this is a modular lattice (\cite{CLP91}), i.e. we have the identity
\[R \vee (S\wedge T) = (R\vee S) \wedge T\]
for all equivalence relations $R,S,T$ on $X$ such that $R\leq T$.

An important property of Mal'tsev categories is that the inclusion of the category $\Grpd(\C)$ of internal groupoids into the category $ \MRG(\C)$ of multiplicative reflexive graphs is an isomorphism, and that the truncation functor $\MRG(\C)\to \RG(\C)$ is fully faithful (\cite{CPP92}).

For a variety, this is also equivalent to the existence of a ternary operation $p$ satisfying the equations $p(x,y,y)=x$ and $p(x,x,y)=y$. In particular, the categories of groups, $R$-modules, rings, Lie algebras and $C^*$-algebras are all examples of Mal'tsev categories; other examples include the category of Heyting algebras, any additive category, or the dual of any topos \cite{B96}.

In any regular category, a commutative square
\[\begin{tikzcd} X \ar[r,"g"] \ar[d,"f"'] & Z \ar[d,"h"] \\ Y \ar[r,"j"']& W
\end{tikzcd}\]
of regular epimorphisms is called a \emph{regular pushout} or \emph{double extension} if the canonical morphism $\langle f,g\rangle\colon X\to Y\times_W Z$ is a regular epimorphism (\cite{J91}). These double extensions are stable under pullback along a commutative square in any regular category.

\begin{prop}[\cite{B03}]\label{thm:reg_PO_split}
If $\C$ is a regular Mal'tsev category, then
\begin{itemize}
\item any square of the form
\[\begin{tikzcd} X\ar[r,"g"] \ar[d,"f", shift left] & Z\ar[d,"j", shift left] \\ Y \ar[r,"h"'] \ar[u,shift left, "s"] & W \ar[u,"t", shift left] \end{tikzcd}\]
where $hf=jg$, $gs=th$, $fs=1_Y$, $jt=1_W$ and $g,h$ are regular epimorphisms is a double extension;
\item a square of regular epimorphisms
\[\begin{tikzcd} X\ar[r,"g"] \ar[d,"f"'] & Z\ar[d,"j"] \\ Y \ar[r,"h"'] & W \end{tikzcd}\]
is a double extension if and only if $f(Eq[g])=Eq[h]$, i.e. if and only if the morphism $Eq[g]\to Eq[h]$ (where $Eq[g]$ and $Eq[h]$ denote the kernel pairs of $g$ and $h$ respectively) induced by $f$ and $j$ is a regular epimorphism.\end{itemize}\end{prop}

We can also define a \emph{triple extension} as a commutative cube
\[\begin{tikzcd}[column sep={3.6em,between origins},row sep={2.7em,between origins}]
X\ar[dr,"\alpha "']\ar[dd,"f"']\ar[rr,"g"] & & Z \ar[dr,"\gamma"]\ar[dd,"h", near end] & \\
& X' \ar[rr,"g'",crossing over,near start]& & Z'\ar[dd,"h'"]\\
Y \ar[rr,"p"',near end] \ar[dr,"\beta"'] & & W \ar[dr,"\delta"] & \\
& Y' \ar[from=uu,crossing over,"f'"',near start]\ar[rr,"p'"'] & & W'
\end{tikzcd}\]
for which all faces, as well as the induced commutative square

\[\begin{tikzcd} X \ar[r,"{\langle f,g\rangle}"] \ar[d,"\alpha"'] & Y\times_W Z \ar[d,"\beta\times_\delta \gamma"] \\ X' \ar[r,"{\langle f',g'\rangle}"']& Y'\times_{W'} Z',
\end{tikzcd}\]
are double extensions. Triple extensions satisfy the same properties as in \autoref{thm:reg_PO_split} : in particular, a split cube between double extensions is always a triple extension.

\subsection{Categorical Galois theory and monotone-light factorization systems}

We recall some definitions from \cite{JK94,JK97}.

A \emph{Galois structure} $\Gamma=(\C,\mathcal{X},I,U,\F)$ consists of a category $\C$, a full reflective subcategory $\mathcal{X}$ of $\C$, with reflector $I\colon \C\to\mathcal{X}$ and inclusion $U\colon \mathcal{X}\to \C$, and a class $\F$ of morphisms of $\C$ containing all isomorphisms, stable under pullbacks, closed under composition, and preserved by $I$. We will call the morphisms in $\F$ \emph{extensions}. Let us write, for any object $B$ of $\C$ (resp. of $\mathcal{X}$), $\C\downarrow B$ (resp. $\mathcal{X}\downarrow B$) for the full subcategory of the slice category $\C/B$ (resp. $\mathcal{X}/ B$) consisting of extensions $f\colon X\to B$. Then any arrow $p\colon E\to B$ induces a functor $p^*\colon \C\downarrow B\to \C\downarrow E$ defined by pulling back. If $p$ is an extension, this functor has a left adjoint $p_!$ defined by composition with $p$; the extension $p$ is said to be of \emph{effective $\F$-descent}, or simply a \emph{monadic extension}, if the functor $p^*$ is monadic.

Moreover, the reflector $I$ induces, for every $B$, a functor $I^B\colon \C\downarrow B\to \mathcal{X}\downarrow I(B)$ which maps $f\colon X\to B$ to $I(f)\colon I(X)\to I(B)$; and every such functor has a right adjoint $U^B\colon \mathcal{X}\downarrow I(B)\to \C\downarrow B$, defined for any $g\colon Y\to I(B)$ by the pullback
\[\begin{tikzcd} B\times_{I(B)}Y \pb\ar[d,"{U^B(g)}"']\ar[r]& Y\ar[d,"g"]\\ B\ar[r,"\eta_B"'] & I(B).\end{tikzcd}\]
The object $B$ is then said to be admissible if $U^B$ is fully faithful, which is equivalent to the reflector $I$ preserving all pullback squares of the form above. A Galois structure $\Gamma$ is said to be admissible if every object is admissible.

Given an admissible Galois structure, an extension $f\colon X\to B$ in $\C\downarrow B$ is said to be
\begin{itemize}
	\item\emph{trivial} if it lies in the replete image of $U^B$, or equivalently if the square
	\[\begin{tikzcd} X\ar[r,"\eta_X"] \ar[d,"f"'] & I(X) \ar[d,"I(f)"] \\ B\ar[r,"\eta_B"'] & I(B) \end{tikzcd}\]
	is a pullback;
	\item\emph{central}, or alternatively a \emph{covering}, if there exists a monadic extension $p\colon E\to B$ such that $p^*(f)$ is trivial;
	\item\emph{normal}, if it is a monadic extension and if $f^*(f)$ is a trivial extension (that is, if the projections of the kernel pair of $f$ are trivial).
\end{itemize}

\begin{example}
	If $\C$ is an exact Mal'tsev category, $\mathcal{X}$ is any Birkhoff (i.e. full reflective and closed under quotients and subobjects) subcategory of $\C$, and $\F$ is the class of regular epimorphisms, then the Galois structure $\Gamma$ is admissible, and moreover, every extension is monadic and every central extension is also normal (\cite{JK94}).
	
	When $\C$ is the category of groups and $\mathcal{X}$ the subcategory of abelian groups, the central extensions in this sense are exactly the surjective group homomorphisms whose kernel is included in the center of the domain (\cite{JK94}). More generally, in any exact Mal'tsev category with coequalizers, the central extensions of the Galois structure defined by the subcategory of abelian objects are the extensions such that the Smith-Pedicchio commutator $[Eq[f],\nabla_X]$ is trivial (\cite{G04}).
\end{example}

If $\Gamma$ is a Galois structure where $\F$ is the class of all morphisms in $\C$, admissibility is equivalent to the reflector $I$ being semi-left-exact in the sense of \cite{CHK85}. Any morphism $f\colon X\to B$ in $ \C$ induces a commutative diagram
\[\begin{tikzcd} X \ar[rrd, bend left=15, "\eta_X"] \ar[ddr, bend right=15, "f"']\ar[dr,"f'" description] & & \\ & {B\times_{I(B)} I(X)} \pb \ar[r] \ar[d,"f''"] & I(X) \ar[d,"I(f)"] \\ & B\ar[r,"\eta_B"'] & I(B); \end{tikzcd}\]
when the reflector $I$ is semi-left-exact, it preserves the pullback in this diagram, $I(f')$ is an isomorphism, and $f''$ is a trivial extension by definition. Moreover, in that case the classes $\E$ of morphisms inverted by $I$ and the class $\M$ of trivial extensions are orthogonal to one another, and thus the two classes form a factorisation system $(\E,\M)$ in $\C$ (\cite{CHK85}). The trivial extensions are then stable under pullbacks, but the class $\E$ does not have this property in general. In order to obtain a stable factorization system, one can localize $\M$ and stabilize $\E$, as in \cite{CJKP97}; this means that we replace $\E$ by the class $\E'$ of morphisms for which every pullback is in $\E$, and $\M$ by the class $\M^*$ of morphisms $f$ that are locally in $\M$, in the sense that there exists a monadic extension $p$ such that $p^*(f)\in \M$. In the context of Galois Theory these are precisely the central extensions. The two classes $\E'$ and $\M^*$ are orthogonal, but in general they do not form a factorization system. When this is the case, the resulting factorization system is called the \emph{monotone-light factorization system} $(\E',\M^*)$ associated with $\Gamma$.

In the case where $\F$ is no longer the class of all morphisms in $\C$, it need not be true that every morphism admits a $(\E,\M)$-factorization. Nevertheless, this is still true for extensions; it is then natural to extend the notion of factorization system to the case where only \emph{some} morphisms have a factorization. This was done by Chikhladze in \cite{C04} :
\begin{defn}
	If $\C$ is a category and $\F$ a class of morphisms of $\C$ containing the identities, closed under composition, and stable under pullbacks, a \emph{relative factorization system} for $\F$ consists of two classes $\E$ and $\M$ of morphisms such that
	\begin{enumerate}\item $\E$ and $\M$ contain identities and are closed under composition with isomorphisms;
		\item $\E$ and $\M$ are orthogonal to one another;
		\item $\M\subset \F$;
		\item every arrow $f$ in $\F$ can be written as $me$ for some $m\in \M$ and $e\in E$.
	\end{enumerate}
\end{defn}

Then any admissible Galois structure $\Gamma=(\C,\mathcal{X},I,\F)$ yields a relative factorization system for $\F$ with $\E$ and $\M$ consisting of the morphisms inverted by $I$ and the trivial extensions, respectively. When moreover this factorization system can be stabilized, then the stable factorization system $(\E',\M^*)$ (where $\E'$ is the class of all morphisms for which any pullback along an arrow in $\F$ is in $\E$) is called a \emph{relative monotone-light factorization system} for $\F$.

\begin{example}
	If $\C$ is the category of simplicial sets, $\mathcal{X}$ the category of groupoids, $I$ the fundamental groupoid functor, and $\F$ the class of Kan fibrations, then every Kan complex is an admissible object, and the central extensions were called \emph{second order coverings} in \cite{BJ99}.
	
	This Galois structure admits a relative monotone-light factorization system, as shown in \cite{C04}.
\end{example}

\begin{example}
	In a finitely complete category, any object $X$ has a corresponding discrete internal groupoid. This defines a fully faithful functor $D\colon \C\to \Grpd(\C)$. If $\C$ is exact, then this functor admits a semi-left-exact left adjoint $\Pi_0:\Grpd(\C)\to \C$ (\cite{B87}). When $\C$ is moreover Mal'tsev, $\C$ is in fact a Birkhoff subcategory of $\Grpd(\C)$, and the central extensions of the Galois structure $(\Grpd(\C),\C,\Pi_0,\F)$ (where $\F$ is the class of regular epimorphisms) are precisely the regular epimorphic discrete fibrations (\cite{G01}). This Galois structure admits a relative monotone-light factorization system (\cite{CEG18}).
\end{example}

\section{The reflection of simplicial objects into groupoids}

\begin{conv}
	For the remainder of this article, $\C$ will denote a regular Mal'tsev category. For a given simplicial object $(X_n)_{n\geq 0}$ with face morphisms $d_i\colon X_n\to X_{n-1}$ for $n\geq 1$ and $0\leq i\leq n$, we will denote by $D_i$ the kernel pair of $d_i$. When necessary, we will write $D_i^\X$ for the kernel pair of $d_i^\X$.
\end{conv}

Note that $\Simp(\C)$, being a functor category, is also regular Mal'tsev.  

\begin{lemma}\label{lem:reg_po_simp}
	If $\X$ is a simplicial object in $\C$, all the commutative squares given by $d_id_j=d_{j-1}d_i$ for $i<j$ are double extensions. Moreover, if $f\colon \X\to \Y$ is a regular epimorphism of simplicial objects, then the corresponding commutative cubes are triple extensions.
\end{lemma}

\begin{proof}
	If $i<j-1$, then we have a diagram
	\[\begin{tikzcd}X_{n+2}\ar[r,shift left,"d_j"]\ar[d,"d_i"']& X_{n+1}\ar[l, shift left,"s_{j-1}"]\ar[d,"d_i"]\\ X_{n+1} \ar[r,shift left,"d_{j-1}"]& X_{n},\ar[l,shift left,"s_{j-2}"]
	\end{tikzcd}\]
	where the two squares obtained by taking the horizontal arrows pointing to the right and to the left both commute (i.e. $d_id_j=d_{j-1}d_i$ and $d_is_{j-1}=s_{j-2}d_i$). On the other hand, if $j=i+1$, then at least one of the inequalities $1\leq j\leq n+2$ is strict, hence at least one of the diagrams
	\[\begin{tikzcd}
	X_{n+2}\ar[r,shift left,"d_{j}"]\ar[d,"d_i"']& X_{n+1}\ar[l, shift left,"s_j"]\ar[d,"d_i"]\\ X_{n+1} \ar[r,shift left,"d_{j-1}"]& X_{n}\ar[l,shift left,"s_{j-1}"]
	\end{tikzcd}
	\quad
	\begin{tikzcd}
	X_{n+2}\ar[r,"d_{j}"]\ar[d,"d_i",shift left]& X_{n+1}\ar[d,"d_i",shift left]\\ X_{n+1} \ar[u,shift left,"s_{i-1}"]\ar[r,"d_{j-1}"']& X_{n}\ar[u,shift left,"s_{i-1}"]
	\end{tikzcd}\]
	will similarly yield two commutative squares; in any case, the commutative square $d_id_j=d_{j-1}d_i$ is a double extension by \autoref{thm:reg_PO_split}.
	
	Moreover, any morphism $f\colon \X\to \Y$ of simplicial objects has to commute with the face and degeneracies; hence, when $f$ is a regular epimorphism, every square
	\[\begin{tikzcd}X_{n+1}\ar[r,"f_{n+1}"]\ar[d,"d_i"']& Y_{n+1} \ar[d,"d_i"]\\ X_{n} \ar[r,"f_{n}"']& Y_{n}
	\end{tikzcd}\]
	is a double extension. The resulting cube will then always be a split epimorphism between double extensions, hence a triple extension.
\end{proof}

\begin{remark}
	The pullback $X_1\times_{X_0}X_1$ of $d_0$ along $d_1$ coincides with the object of $(2,1)$-horns $\Lambda^2_1 (\X)$, and similarly the other two pullbacks $X_1\times_{X_0}X_1$, which define the kernel pairs of $d_0$ and $d_1$, coincide with the objects of $(2,2)$ and $(2,0)$-horns, respectively. In particular, \autoref{lem:reg_po_simp} shows that every simplicial object satisfies the Kan property and that every regular epimorphism is a Kan fibration for $2$-horns. The proof for the higher order horns can be done in the same way, using $n$-fold extensions for $n\geq 3$, as in \cite{EGoV12}.
\end{remark}

As a consequence we have
\begin{coro}
	If $f\colon \X\to \Y$ is a regular epimorphism in $\Simp(\C)$, then $f(D_i^\X \wedge D_j^\X)=D_i^\Y \wedge D_j^\Y$. Moreover, for any $i<j<k$, we have
	\[d_k(D_i\wedge D_j)=D_i\wedge D_j\]
	\[d_j(D_i\wedge D_k)=D_i\wedge D_{k-1}\]
	\[d_i(D_j\wedge D_k)=D_{j-1}\wedge D_{k-1}.\]
\end{coro}

\begin{proof}
	By \autoref{lem:reg_po_simp} $f\colon \X\to \Y$ induces a triple extension. In particular the square
	\[\begin{tikzcd}X_{n}\ar[r,"{\langle d^\X_i,d^\X_j\rangle}"]\ar[d,"f_n"']& X_{n-1}\times_{X_{n-2}} X_{n-1} \ar[d] \\ Y_{n} \ar[r,"{\langle d^\Y_i,d^\Y_j\rangle}"'] & Y_{n-1}\times_{Y_{n-2}} Y_{n-1}
	\end{tikzcd}\]
	is a double extension, so that
	\[ f(D_i^\X \wedge D_j^\X) =f(Eq[\langle d^\X_i,d^\X_j\rangle]) = \langle d^\Y_i,d^\Y_j\rangle = D_i^\Y \wedge D_j^\Y.\]
	Moreover, $d_k$ is a component of the regular epimorphism $\epsilon_\X \colon Dec(\X) \to \X$, and thus the cube
	\[\begin{tikzcd}[column sep={3.6em,between origins},row sep={2.7em,between origins}] X_{n+3}\ar[dr,"d_k"']\ar[dd,"d_i"']\ar[rr,"d_j"] & & X_{n+2} \ar[dr,"d_{k-1}"]\ar[dd,"d_i", near end] & \\
	& X_{n+2} \ar[rr,"d_j",crossing over,near start]& & X_{n+1}\ar[dd,"d_i"]\\
	X_{n+2} \ar[rr,"d_{j-1}",near end] \ar[dr,"d_{k-1}"'] & & X_{n+1} \ar[dr,"d_{k-2}"] & \\
	& X_{n+1} \ar[from=uu,crossing over,"d_i"',near start]\ar[rr,"d_{j-1}"'] & & X_n\end{tikzcd}\]
	is a triple extension; in particular the squares
	\[\begin{tikzcd}X_{n+3}\rar["{\langle d_i,d_j\rangle}"]\ar[d,"d_k"']& X_{n+2}\times_{X_{n+1}} X_{n+2} \ar[d]\\ X_{n+2} \ar[r,"{\langle d_i,d_j\rangle}"'] & X_{n+1}\times_{X_{n}} X_{n+1}
	\end{tikzcd},
	\quad
	\begin{tikzcd}X_{n+3}\ar[r,"{\langle d_i,d_k\rangle}"]\ar[d,"d_j"']& X_{n+2}\times_{X_{n+1}} X_{n+2} \ar[d]\\ X_{n+2} \ar[r,"{\langle d_i,d_{k-1}\rangle}"'] & X_{n+1}\times_{X_{n}} X_{n+1}
	\end{tikzcd}\]
	and
	\[\begin{tikzcd}
	X_{n+3}\ar[r,"{\langle d_j,d_k\rangle}"]\ar[d,"d_i"']& X_{n+2}\times_{X_{n+1}} X_{n+2} \ar[d]\\ X_{n+2} \ar[r,"{\langle d_{j-1},d_{k-1}\rangle}"'] & X_{n+1}\times_{X_{n}} X_{n+1}
	\end{tikzcd}\]
	are all double extensions, which implies the desired equalities.
\end{proof}

\begin{lemma}\label{lem:H_1}
	For any simplicial object $\X$, the following equivalence relations in $X_1$ are all equal :
	\[d_0(D_1\wedge D_2)=d_1(D_0\wedge D_2)=d_2(D_0\wedge D_1).\]
\end{lemma}

\begin{proof}
	We prove the first identity; the other one is obtained in a similar way. Since $D_1\wedge D_2=d_2(D_1\wedge D_3)$ and $d_0(D_1\wedge D_3)=D_0\wedge D_2$, we have
	\[d_0(D_1\wedge D_2)=d_0(d_2(D_1\wedge D_3))=d_1(d_0(D_1\wedge D_3))=d_1(D_0\wedge D_2).\]
\end{proof}

\begin{defn}
	We will call $H_1(\X)$ the equivalence relation $d^\X_1(D_0^\X \wedge D_2^\X )$.
\end{defn}

\begin{prop}\label{prop:grpds_in_simp}
	Let $\X$ be a simplicial object in $\C$. Then for all $n\geq 2$ the following conditions are equivalent :
	\begin{enumerate}
		\item $D_i\wedge D_j=\Delta_{X_n}$ for all $0\leq i<j\leq n$;
		\item $D_0\wedge D_n=\Delta_{X_n}$;
		\item there exist $0\leq i < j\leq n$ such that $D_i\wedge D_j=\Delta_{X_n}$.
	\end{enumerate}
Moreover, $\X$ is an internal groupoid if and only if it satisfies these conditions for all $n\geq 2$.
\end{prop}

\begin{proof}
	It is clear that $(1)$ implies $(2)$ and that $(2)$ implies $(3)$; we prove that the third implies the first by induction. We first consider the case where $n=2$; if $D_i\wedge D_j=\Delta_{X_2}$, and $k$ is such that $\{0,1,2\}=\{i,j,k\}$, we need to prove that $D_i\wedge D_k=\Delta_{X_2}$. We have $d_i(D_i\wedge D_k)=\Delta_{X_1}$ and
	\[d_j(D_i\wedge D_k)=d_k(D_i\wedge D_j)=\Delta_{X_1}\]
	by \autoref{lem:H_1}, and thus $D_i\wedge D_k\leq D_i\wedge D_j=\Delta_{X_2}$.
	
	Assuming now that the condition holds for $n$, we prove that it holds for $n+1$. Assume that $D_i\wedge D_j=\Delta_{X_{n+1}}$; then taking images by $d_k$ (for $k\notin \{i,j\}$) on both sides shows that $D_{i'}\wedge D_{j'}=\Delta_{X_n}$ for some $i',j'$, and thus, by the induction hypothesis, for all $i',j'$. In particular, for any $0\leq r<s\leq n+1$, we have for some $r',s'$
	\[d_i(D_r\wedge D_s) \leq D_{r'}\wedge D_{s'}=\Delta_{X_n},\]
	so that $D_r\wedge D_s\leq D_i$; and similarly $D_r\wedge D_s\leq D_j$, so that $D_r\wedge D_s=\Delta_{X_{n+1}}$.
	
	Now $\X$ is an internal groupoid if and only if the squares
	\[\begin{tikzcd} X_{n}\ar[r,"d_0"]\ar[d,"d_{n}"']& X_{n-1}\ar[d,"d_{n-1}"]\\ X_{n-1}\ar[r,"d_{0}"']& X_{n-2}\end{tikzcd}\]
	are all pullbacks. Since we know already that they are all double extensions, this is equivalent to the fact that the pair $d_0,d_n$ is jointly monic, and this is equivalent to $(2)$. Thus any internal category always satisfies the second condition, and conversely any simplicial object satisfying the first one is an internal category where the square
	\[\begin{tikzcd} X_{2}\ar[r,"d_0"]\ar[d,"d_{1}"']& X_{1}\ar[d,"d_{0}"]\\ X_{1}\ar[r,"d_{0}"']& X_{0}\end{tikzcd}\]
	is a pullback. This condition is equivalent to the internal category being a groupoid.
\end{proof}

Note that in the above proof we only needed to know that $\X$ was an internal category to prove that it satisfied the conditions; so this gives us a new proof of the fact that any internal category in a regular Mal'tsev category is an internal groupoid.

\begin{coro}\label{coro:closednessqs}
	The subcategory $\Grpd(\C)$ of $\Simp(\C)$ is closed under quotients and subobjects.
\end{coro}

\begin{proof}
All the intersections that characterize internal groupoids in \autoref{prop:grpds_in_simp} are preserved by regular epimorphisms of simplicial objects, which shows that groupoids are closed under quotients. Furthermore, they are also closed under subobjects; indeed, if $m\colon \X\to \Y$ is a monomorphism in $\Simp(\C)$ with $\Y$ a groupoid, then for any $0\leq i<j\leq n$, the cube induced by the identity $d_id_j=d_{j-1}d_i$ and $m$ yields a commutative square
\[\begin{tikzcd}
X_n\ar[r,"m_n"] \ar[d,"{\langle d_i,d_j\rangle}"'] & Y_n \ar[d,"{\langle d_i, d_j \rangle}"]\\ X_{n-1}\times_{X_{n-2}} X_{n-1} \ar[r,"\widetilde{m}"']  & Y_{n-1}\times_{Y_{n-2}} Y_{n-1} \end{tikzcd}\]
where the horizontal arrows are monomorphisms and the right-hand vertical side is an isomorphism, and thus the left-hand vertical side is a monomorphism. Since it is also a regular epimorphism (by \autoref{lem:reg_po_simp}), this means $\langle d_i,d_j\rangle$ is an isomorphism; hence $\X$ is an internal groupoid.
\end{proof}

\begin{remark}
In fact \autoref{coro:closednessqs} also characterizes Mal'tsev categories among the regular (or even finitely complete) ones : indeed a reflexive relation $R\hookrightarrow X\times X$ is just a subobject of the reflexive graph $(X\times X,X,\pi_1,\pi_2,\delta_X)$, and by taking iterated simplicial kernels, one can extend this to a monomorphism in $\Simp(\C)$, whose codomain is just the nerve of the indiscrete equivalence relation/groupoid on $X$. Thus every reflexive relation is a subobject of a groupoid, and a relation is a groupoid if and only if it is an equivalence relation. Accordingly :
\end{remark}

\begin{coro}A regular category $\C$ is a Mal'tsev category if and only if $\Grpd(\C)$ is closed under subobjects in $\Simp(\C)$.
\end{coro}

\begin{conv}
For the remainder of this article, we assume that the category $\C$ is exact.
\end{conv}

In this setting, we have

\begin{thm}\label{thm:adjoint}
	If $\X$ is a simplicial object in $\C$, then the quotient $\frac{X_1}{H_1(\X)}$ and the object $X_0$ admit a groupoid structure
	\begin{equation}\label{eq:structure}
		\begin{tikzcd}
			\frac{X_1}{H_1(\X)} \ar[r,shift left=2.2, "\overline{d_0}"] \ar[r,shift right=2.2,"\overline{d_1}"'] & \ar[l,"\overline{s_0}"description] X_0 ;
		\end{tikzcd}
	\end{equation}
	and taking the nerve of this groupoid defines a functor $\Pi_1\colon \Simp(\C)\to\Grpd(\C)$, which is left adjoint to the inclusion $\Grpd(\C)\to \Simp(\C)$.
	
	In particular, $\Grpd(\C)$ is a Birkhoff subcategory of $\Simp(\C)$.
\end{thm}

\begin{proof}
Note first that since by definition $H_1(\X)\leq D_0\wedge D_1$, $d_0$ and $d_1\colon X_1\to X_0$ factor through the coequalizer $\eta_1$ of $H_1(\X)$ as $\overline{d_0}\eta_1$ and $\overline{d_1}\eta_1$ respectively, and $\overline{d_0}$ and $\overline{d_1}$ have a common section $\eta_1s_0$, which we will also denote $\overline{s_0}$, so that we get a morphism of reflexive graphs
\[\begin{tikzcd} X_1 \ar[rr,"\eta_1"] \ar[dr,shift left,"d_1"] \ar[dr,shift right,"d_0"'] & & \frac{X_1}{H_1(\X)} \ar[dl,shift left,"\overline{d_1}"] \ar[dl,shift right,"\overline{d_0}"'] \\ & X_0. &\end{tikzcd}\]

Let us then form the pullback
\[\begin{tikzcd}\frac{X_1}{H_1(\X)}\times_{X_0}\frac{X_1}{H_1(\X)} \pb \ar[d,"\overline{d_0}"']\ar[r,"\overline{d_2}"] & \frac{X_1}{H_1(\X)}\ar[d,"\overline{d_0}"]\\ \frac{X_1}{H_1(\X)}\ar[r,"\overline{d_1}"']& X_0; \end{tikzcd}\]
to prove that the reflexive graph \eqref{eq:structure} is a groupoid, it suffices to prove that there exists a morphism $\overline{d_1}\colon \frac{X_1}{H_1(\X)} \times_{X_0} \frac{X_1}{H_1(\X)}$ that satisfies the relevant identities.

By Proposition 4.1 in \cite{B03}, $\eta_1\times \eta_1\colon X_1\times_{X_0}X_1\to \frac{X_1}{H_1}\times_{X_0}\frac{X_1}{H_1}$ is a regular epimorphism, and as a consequence so is
\[\langle \eta_1 d_0,\eta_1d_2\rangle=(\eta_1\times \eta_1)\circ \langle d_0,d_2\rangle \colon X_2\to \frac{X_1}{H_1(\X)}\times_{X_0}\frac{X_1}{H_1(\X)},\]
which we will denote $\eta_2$. We also define $H_2(\X)=Eq[\eta_2]$. Now to prove the existence of $\overline{d_1}$, we need to show that $\eta_1d_1\colon X_2\to \frac{X_1}{H_1(\X)}$ factorizes through $\eta_2$; for this it is enough to show that $\eta_1d_1(H_2(\X))=\Delta_{X_2}$, which is equivalent to $d_1(H_2(\X))\leq H_1(\X)$. Since $\overline{d_0}$ and $\overline{d_2}$ are jointly monic by construction, we find that
\begin{align*}H_2(\X) & =d_0^{-1}(H_1(\X))\wedge d_2^{-1}(H_1(\X)) \\ & = d_0^{-1}(d_0(D_1\wedge D_2))\wedge d_2^{-1}(d_2(D_0\wedge D_1))\\ & =(D_0\vee(D_1\wedge D_2))\wedge (D_2\vee (D_0\wedge D_1)).\end{align*}
Using the modularity of the lattice of equivalence relations on $X_2$, one sees that this is equal to
\[((D_0\vee(D_1\wedge D_2))\wedge D_2)\vee (D_0\wedge D_1) = (D_0\wedge D_2)\vee(D_1\wedge D_2)\vee (D_0\wedge D_1).\]
From this last expression, we get that $d_1(H_2(\X))=d_1(D_0\wedge D_2)=H_1(\X)$. This proves the existence of $\overline{d_1}$ such that $\overline{d_1}\eta_2=\eta_1d_1$. Let us also denote $\overline{s_0}$ the unique morphism $\frac{X_1}{H_1(\X)}\to \frac{X_1}{H_1(\X)}\times_{X_0}\frac{X_1}{H_1(\X)}$ such that $\overline{d_0}\overline{s_0}=1_{\frac{X_1}{H_1(\X)}}$ and $\overline{d_2}\overline{s_0}=\overline{s_0}\overline{d_1}$, and $\overline{s_1}$ the unique morphism $\frac{X_1}{H_1(\X)}\to \frac{X_1}{H_1(\X)}\times_{X_0}\frac{X_1}{H_1(\X)}$ such that $\overline{d_2}\overline{s_1}=1_{\frac{X_1}{H_1(\X)}}$ and $\overline{d_0}\overline{s_1}=\overline{s_0}\overline{d_0}$. Using the fact that $\eta_1$ and $\eta_2$ are regular epimorphisms, one can now easily prove that all the simplicial identities are satisfied. This endows the quotient graph with the structure of a multiplicative graph, which is then automatically a groupoid, which we denote $\Pi_1(\X)$. We also denote $\eta_\X\colon \X\to \Pi_1(\X)$ the morphism of simplicial objects induced by $1_{X_0}$, $\eta_1$ and $\eta_2$. We can show that $\eta_n$ is is a regular epimorphism for all $n$, by iterating the argument showing that $\eta_2$ is a regular epimorphism.

It remains to prove that $\Pi_1$ is indeed a left adjoint for the inclusion $\Grpd(\C)\to \Simp(\C)$. For this we must prove that for every morphism $f\colon \X\to \Y$ to a groupoid $\Y$, there exists a factorization of $f_n$ through $\eta_n\colon X_n\to \Pi_1(\X)_n$ for all $n$ (note that such a factorization is unique, as every $\eta_n$ is a regular epimorphism). The case $n=0$ is trivial as $\eta_0$ is the identity. For $n=1$, it is enough to prove that $Eq[f_1]\geq H_1(\X)$, or equivalently $f_1(H_1(\X))=\Delta_{Y_1}$. Now
\[f_1(H_1(\X))=f_1d^\X_1(D_0^\X \wedge D_2^\X )=d^\Y_1f_2(D_0^\X \wedge D_2^\X )\leq d^\Y_1(D^\Y_0\wedge D^\Y_2)=\Delta_{Y_1}.\]

This shows that the truncation of $f$ to a morphism $(f_1,f_0)$ of reflexive graph factors through the groupoid $X_1/H_1(\X)$, with a factorization $(g_1,g_0=f_0)$; applying the nerve functor allows us to extend this factorization to higher levels, resulting in morphisms $g_n\colon \Pi_1(\X)_n\to Y_n$. Then the factorizations $f_n=g_n\eta_n$, for $n\geq 2$, can be obtained from the universal property of the pullbacks defining each $\overline{X}_n$ and $Y_n$. Then since each $\eta_n$ is a regular epimorphism, the morphisms $g_n$ define a morphism of simplicial objects.
\end{proof}

Let us denote $H_n(\X)$ the kernel pair of $\eta_n$. We have proved already that $H_2(\X)=(D_0\wedge D_1)\vee (D_0\wedge D_2)\vee (D_1\wedge D_2)$. For the next section, it will be useful to prove a similar formula for $H_n(\X)$, for $n\geq 3$ :

\begin{prop} For all $n\geq 2$, we have
\[H_n(\X)=\bigvee_{0\leq i<j\leq n} (D_i\wedge D_j).\]
\end{prop}

\begin{proof}
We prove the result by induction on $n$. The case $n=2$ was done in the proof of \autoref{thm:adjoint}. Now let us assume that it holds for $n$; since by construction the square
\[\begin{tikzcd} \Pi_1(\X)_{n+1} \pb \ar[r,"d_0"] \ar[d,"d_{n+1}"']& \Pi_1(\X)_n \ar[d,"d_n"] \\ \Pi_1(\X)_n \ar[r,"d_0"'] & \Pi_1(\X)_{n-1}\end{tikzcd}\]
is a pullback, so that the two morphisms $d_0,d_{n+1}$ are jointly monic, we have for $n+1$
\begin{align*}H_{n+1}(\X) & =Eq[\eta_{n+1}]=Eq[d_{0}\eta_{n+1}]\wedge Eq[d_{n+1}\eta_{n+1}] \\ & = Eq[\eta_n d_0]\wedge Eq[\eta_{n}d_{n+1}] = d_0^{-1}(H_n(\X))\wedge d_{n+1}^{-1}(H_n(\X))\end{align*}
Moreover, by the induction hypothesis we have the identities
\[H_n(\X)=\smashoperator[r]{\bigvee_{0\leq i<j\leq n}} (D_i\wedge D_j)=  d_0\left(\smashoperator[r]{\bigvee_{0<i<j\leq n+1}} (D_i\wedge D_j)\right)\]
and
\[H_n(\X)=\smashoperator[r]{\bigvee_{0\leq i<j\leq n}} (D_i\wedge D_j)=  d_{n+1}\left(\smashoperator[r]{\bigvee_{0\leq i<j < n+1}} (D_i\wedge D_j)\right).\]
Combining all these, we get the identity
\[H_{n+1}(\X)=\left(D_0\vee \smashoperator{\bigvee_{0<i<j\leq n+1}} (D_i\wedge D_j)\right) \wedge \left( D_{n+1}\vee \smashoperator{\bigvee_{0\leq i<j < n+1}} (D_i\wedge D_j)\right).\]
From there we already see that
\[H_{n+1}(\X)\geq \smashoperator[r]{\bigvee_{0\leq i<j\leq n+1}}(D_i\wedge D_j).\]

For the converse inequality, first note that
\[\smashoperator[r]{\bigvee_{0<i<j\leq n+1}} (D_i\wedge D_j)\leq \left(D_{n+1}\vee \left(\smashoperator[r]{\bigvee_{0\leq i<j< n+1}} (D_i\wedge D_j)\right)\right),\]
and thus, since the lattice of equivalence relations of $X_{n+1}$ is modular, we have
\[H_{n+1}(\X)= \smashoperator[r]{\bigvee_{0<i<j\leq n+1}} (D_i\wedge D_j) \vee \left(D_0\wedge \left( D_{n+1}\vee \smashoperator{\bigvee_{0 \leq i<j< n+1}} (D_i\wedge D_j)\right) \right).\]
Now to conclude the proof it is enough to prove that
\begin{equation}D_0\wedge \left( D_{m}\vee \smashoperator{\bigvee_{0\leq i<j< m}} (D_i\wedge D_j)\right) = \smashoperator[r]{\bigvee_{0< j \leq m}}(D_0\wedge D_j)\label{eq:step}\end{equation}
for all $m\geq 1$, which we will do by induction. The case where $m=1$ is trivial, so let us now assume that \eqref{eq:step} holds for some $m$. Then we have
\begin{align*}
d_{m} & \left( D_0 \wedge \left(D_{m+1}\vee \smashoperator{\bigvee_{0\leq i<j< m+1}} (D_i\wedge D_j)\right)\right)
\\ & \leq d_{m}(D_0)\wedge \left( d_m(D_{m+1})\vee\smashoperator{\bigvee_{0\leq i<j< m+1}} d_m(D_i\wedge D_j)\right)
\\ & = D_0\wedge \left( D_{m}\vee \smashoperator{\bigvee_{0\leq i<j< m}}(D_i\wedge D_j)\right)
\\ & = \smashoperator[r]{\bigvee_{0 < j\leq m}} (D_0\wedge D_j)
\\ & = d_m\left(\smashoperator[r]{\bigvee_{0< j\leq m+1}}(D_0\wedge D_j)\right),
\end{align*}
and as a consequence we have
\[D_0 \wedge \left(D_{m+1}\vee \smashoperator{\bigvee_{0\leq i<j< m+1}} (D_i\wedge D_j)\right)\leq D_m\vee \smashoperator{\bigvee_{0< j\leq m+1}}(D_0\wedge D_j).\]
It follows that the left-hand side must be equal to
\[D_0\wedge \left( D_{m+1}\vee \smashoperator{\bigvee_{0\leq i<j< m+1}} (D_i\wedge D_j)\right) \wedge \left(D_{m}\vee \smashoperator{\bigvee_{0< j\leq m+1}}(D_0\wedge D_j)\right).\]
Now since $ \bigvee_{0< j\leq m+1}(D_0\wedge D_j)\leq D_0$, using again the modularity law, we find that
\[D_0 \wedge \left(D_m\vee \smashoperator{\bigvee_{0< j\leq m+1}}(D_0\wedge D_j)\right) =\smashoperator[r]{\bigvee_{0< j\leq m+1}}(D_0\wedge D_j),\]
and this is smaller than $\left( D_{m+1}\vee \bigvee_{0\leq i<j< m+1} (D_i\wedge D_j)\right)$, which concludes the proof.
\end{proof}

\begin{remark}If the category $\mathcal{C}$ is not only exact Mal'tsev but also arithmetical (\cite{P96}), then the category $\Grpd(\C)$ coincides with the category of equivalence relations, which is thus a Birkhoff subcategory of $\Simp(\C)$. Note that in that case, $H_1(\X)=d_0(D_1\wedge D_2)=D_0\wedge D_1$, since direct images preserve intersections of equivalence relations (by Theorem 5.2 of \cite{B05}). Accordingly our reflection becomes a reflection of $\Simp(\C)$ into $\Eq(\C)$.

Since every groupoid is a quotient of an equivalence relation, $\Eq(\C)$ is closed under quotients in $\Simp(\C)$ if and only if $\Eq(\C)=\Grpd(\C)$.
\end{remark}

\begin{coro}An exact Mal'tsev category is arithmetical if and only if $\Eq(\C)$ is a Birkhoff subcategory of $\Simp(\C)$.\end{coro}

\begin{remark}Note that, in contrast with the Smith-Pedicchio commutator, which yields a left adjoint of the forgetful/inclusion functor $\Grpd(\C)\to \RG(\C)$ (\cite{P95}), we don't need to assume the existence of any colimits to define $H_1(\X)$.
\end{remark}

\section{Characterization of central extensions}

Being a Birkhoff subcategory of the exact Mal'tsev category $\Simp(\C)$, $\Grpd(\C)$ is admissible in the sense of categorical Galois theory, when $\F$ is the class of all regular epimorphisms. In this section we will characterize the central extensions with respect to this reflection.

\begin{conv} If $f\colon \X\to \Y$ is a morphism in $\Simp(\C)$, we will denote $F_n$ the kernel pair of the corresponding morphism $f_n\colon X_n\to Y_n$, for all $n\geq 0$. Similarly, for morphisms $g\colon \Z\to \W$ and $f'\colon \X'\to \Y'$ in $\Simp(\C)$, we will denote the corresponding kernel pairs $G_n$ and $F'_n$ (for $n\geq 0$), respectively.
\end{conv}

First, we note that Proposition 4.2 of \cite{JK94} implies, in our case, that trivial extensions $f\colon \X\to \Y$ are characterized by the property that $F_n\wedge H_n(\X)=\Delta_{X_n}$ for all $n\geq 0$, that is :

\[F_n\wedge \left(\smashoperator[r]{\bigvee_{0\leq i<j\leq n}} D_i\wedge D_j\right)=\Delta_{X_n}\]
for $n\geq 2$ and
\[F_1\wedge d_0(D_1\wedge D_2)=\Delta_{X_1}.\]

Our characterization of central extensions is then obtained simply by "distributing" the intersection with $F_n$ appearing in these equations with the join or image. In other words we have

\begin{thm}\label{thm:char_central_exts}
A regular epimorphism $f\colon \X \to \Y$ is a central extension with respect to the Galois structure induced by the reflection of $\Simp(\C)$ into $\Grpd(\C)$ if and only if
\begin{equation}\label{eq:central_1}d_1(F_1\wedge D_0\wedge D_2) =\Delta_{X_1}\end{equation}
and, for all $n\geq 2$ and $i,j$ such that $0\leq i<j \leq n$,
\begin{equation}\label{eq:central_2} F_n\wedge D_i\wedge D_j =\Delta_{X_n}.\end{equation}
\end{thm}

To prove this we will need a couple of lemmas.

\begin{lemma}\label{thm:pb_stab}
Let
\[\begin{tikzcd} \bP \pb \ar[r,"g'"] \ar[d,"f'"'] & \X \ar[d,"f"]\\ \Z \ar[r,"g"'] & \Y \end{tikzcd}\]
be a pullback square of regular epimorphisms in $\Simp(\C)$, and let $n\geq 2$ and $0\leq i<j\leq n$. Let us denote $d_i'$ the face morphisms of the simplicial object $\bP$, and $D_i'$ their kernel pairs. Then
\[F_n\wedge D_i\wedge D_j =\Delta_{X_n} \Leftrightarrow F_n'\wedge D_i'\wedge D_j' =\Delta_{P_n}.\]
\end{lemma}

\begin{proof}Since pullbacks in $\Simp(\C)$ are computed "levelwise" in $\C$, for all $n$ the square
\[\begin{tikzcd} P_n \pb\ar[r,"g_n'"] \ar[d,"f_n'"'] & X_n \ar[d,"f_n"]\\ Z_n \ar[r,"g_n"'] & Y_n \end{tikzcd}\]
is a pullback. Therefore, in the cube
\[\begin{tikzcd}[column sep={5.4em,between origins},row sep={3em,between origins}]
P_n\ar[dr,"f_n'"']\ar[dd,"{( d_i',d_j')}"']\ar[rr,"g_n'"] & & X_n \ar[dr,"f_n"]\ar[dd,"{( d_i,d_j)}", near end] & \\
& Z_n \ar[rr,"g_n",crossing over,near start]& & F\ar[dd]\\
P_{n-1}\times_{P_{n-2}} P_{n-1} \ar[rr] \ar[dr] & & X_{n-1}\times_{X_{n-2}} X_{n-1} \ar[dr,] & \\
& Z_{n-1}\times_{Z_{n-2}} Z_{n-1}\ar[from=uu,crossing over]\ar[rr] & & Y_{n-1}\times_{Y_{n-2}} Y_{n-1}\end{tikzcd}\]
the top and bottom faces are pullbacks; one can then show that the square
\[\begin{tikzcd} P_n \ar[d,"g_n'"'] \ar[r]& (P_{n-1}\times_{P_{n-2}} P_{n-1})\times_{Z_{n-1}\times_{Z_{n-2}} Z_{n-1}} Z_n \ar[d] \\ X_n\ar[r]& (X_{n-1}\times_{X_{n-2}} X_{n-1})\times_{Y_{n-1}\times_{Y_{n-2}} Y_{n-1}} Y_n \end{tikzcd}\]
is a pullback, which implies that
\begin{equation} g_n'(F_n'\wedge D_i'\wedge D_j')=F_n\wedge D_i\wedge D_j.\label{eq:img_inter}\end{equation}

In particular, $F_n'\wedge D_i'\wedge D_j'=\Delta_{P_n}$ implies that $F_n\wedge D_i\wedge D_j=\Delta_{X_n}$.

For the converse, the equation \eqref{eq:img_inter} shows already that if $F_n\wedge D_i\wedge D_j=\Delta_{X_n}$, then $F_n'\wedge D_i'\wedge D_j'\leq G_n'$. Since it is also smaller than $F_n'$, and since $f_n'$ and $g_n'$ are jointly monic by construction, we have $F_n'\wedge D_i'\wedge D_j'=\Delta_{P_n}$.
\end{proof}

\begin{lemma}\label{thm:split}
Let $f\colon X\to Y$ be a split epimorphism, with section $s\colon Y\to X$, and let $A,B$ be two equivalence relations on $X$, with respective coequalizers $q_A,q_B$. Assume that we have a diagram
\begin{equation}\label{eq:split_epi}\begin{tikzcd} X/A \ar[d,shift left]& X \ar[l,"q_A"'] \ar[d,shift left,"f"] \ar[r,"q_B"] & X/B \ar[d, shift left] \\ Y/A' \ar[u, shift left] & Y \ar[r,"q_{B'}"'] \ar[l,"q_{A'}"] \ar[u, shift left,"s"]& Y/B' \ar[u, shift left]\end{tikzcd}\end{equation}
where the vertical downwward arrows are split epimorphisms, and the upward and downward squares commute. Then the following conditions are equivalent :
\begin{enumerate}\item $Eq[f]\wedge (A\vee B)=\Delta_{X}$
\item $Eq[f]\wedge A=\Delta_X = Eq[f]\wedge B$.
\end{enumerate}
\end{lemma}

\begin{proof}
	First of all, we have the inequality
	\[(Eq[f]\wedge A)\vee (Eq[f]\wedge B) \leq Eq[f]\wedge (A\vee B),\]
	which immediately proves that the first condition implies the second.
	
	For the converse, we can complete the diagram \eqref{eq:split_epi} by taking the pushouts of the top and bottom spans. This yields a cube
	\[\begin{tikzcd}[column sep={4em,between origins},row sep={3em,between origins}]
	X\ar[dr,"f",shift left]\ar[dd,"q_A"']\ar[rr,"q_B"] & & X/B \ar[dr,shift left]\ar[dd] & \\
	& Y \ar[rr,"q_{B'}" near start,crossing over] \ar[ul,"s",shift left]& & Y/B'\ar[dd]\ar[ul, shift left]\\
	X/A \ar[rr] \ar[dr,shift left] & & X/(A\vee B) \ar[dr, shift left] & \\
	& Y/A' \ar[from=uu,crossing over,"q_{A'}"',near start]\ar[rr] \ar[ul,shift left] & & Y/(A'\vee B')\ar[ul,shift left]\end{tikzcd}\]
	which is a split epimorphism between double extensions, hence a triple extension. In particular, the square
	\[\begin{tikzcd} X \ar[r,"q_B"] \ar[d,"{\langle q_A,f \rangle}"] & X/B \ar[d,"{\gamma}"] \\ {X/A\times_{Y/A'} Y} \ar[r] & {X/(A\vee B) \times_{Y/(A'\vee B')} Y/B'}\end{tikzcd}\]
	is a double extension. Assume now that $Eq[f]\wedge A=\Delta_X = Eq[f]\wedge B$. The first equality implies that $\langle q_A,f\rangle$ is a monomorphism, hence an isomorphism; then so is $\gamma$ in the diagram above, and thus the left and right faces of the cube are pullbacks. Similarly, the second equality implies that the top face is a pullback as well, and then so is the square
	\[\begin{tikzcd} X \ar[d,"f"'] \ar[r,"q_{A\vee B}"] & X/(A\vee B) \ar[d] \\ Y\ar[r] & Y/(A'\vee B'), \end{tikzcd}\]
	since it is the composite of the top and right faces. This implies that $Eq[f]\wedge (A\vee B)=\Delta_{X}$.
\end{proof}

\begin{proof}[Proof of \autoref{thm:char_central_exts}]
	Let us consider the diagram
	\[\begin{tikzcd}\Pi_1(\X\times_\Y\X)\ar[d,"\Pi_1(\pi_1)"'] & \ar[l,"\eta_{\X\times_\Y\X}"'] \X\times_\Y\X \ar[d,"\pi_1"]\ar[r,"\pi_2"] & \X \ar[d,"f"] \\ \Pi_1(\X) & \ar[l,"\eta_\X"] \X \ar[r,"f"'] & \Y. \end{tikzcd}\]
	Now assume first that $f$ is a central extension, so that the left-hand square is a pullback. Since by construction $\Pi_1(\X\times_\Y\X)$ is an internal groupoid, \eqref{eq:central_2} holds for $\Pi_1(\pi_1)$, and then by \autoref{thm:pb_stab} it also holds for $\pi_1$ and thus for $f$.
	
	Assuming now that \eqref{eq:central_2} holds, then again by \autoref{thm:pb_stab} it also holds with $\pi_1\colon\X\times_\Y\X\to \X$, so that for all $i,j$ such that $0\leq i<j \leq n$,
	\[Eq[(\pi_1)_n]\wedge D'_i\wedge D'_j =\Delta_{X_n\times_{Y_n}X_n}.\]
	But $\pi_1$ is a split epimorphism in the category of simplicial objects of $\C$. Thus in particular, for all $0\leq i<j \leq n$, $(\pi_1)_n$ and $D_i'\wedge D_j'$ satisfy the assumptions of \autoref{thm:split}, and thus we have
	\[Eq[(\pi_1)_n]\wedge H_n(\X \times_{\Y} \X)=Eq[(\pi_1)_n]\wedge \left(\smashoperator[r]{\bigvee_{0\leq i<j\leq n}} D_i'\wedge D_j'\right)=\Delta_{X_n\times_{Y_n}X_n}.\]
	This implies that the left-hand square is a pullback; thus $\pi_1$ is a trivial extension, and $f$ is a central extension.
\end{proof}

The equivalence relation $F_2\wedge D_0\wedge D_1$ is the kernel pair of the arrow $\theta^2_2 \colon X_2\to \Lambda^2_2(\X)\times_{\Lambda^2_2(\Y)} Y_2$ defined as in \eqref{eq:kan_diag}. Since $\theta^2_2$ is a regular epimorphism in $\C$ whenever $f$ is in $\F$, $F_2\wedge D_0\wedge D_1$ is trivial if and only if the square
\[\begin{tikzcd} X_2 \ar[r, "f_2"] \ar[d, "\lambda^2_2"'] & Y_2 \ar[d,"\lambda^2_2"] \\ \Lambda^2_2(\X)\ar[r,"\Lambda^2_2(f)"'] & \Lambda^2_2(\Y)\end{tikzcd}\]
is a pullback. The triviality of $F_2\wedge D_0\wedge D_2$ and $F_2\wedge D_1\wedge D_2$ can be interpreted in the same way with the horn objects $\Lambda^2_1$ and $\Lambda^2_0$.

Moreover, the higher order conditions $F_n\wedge D_i\wedge D_j=\Delta_{X_n}$ imply that all the morphisms $\theta^n_k$, for $n\geq 2$, are isomorphisms, and thus that all squares
\[\begin{tikzcd} X_n \ar[r, "f_n"] \ar[d, "\lambda^n_k"'] & Y_n \ar[d,"\lambda^n_k"] \\ \Lambda^n_k(\X)\ar[r,"\Lambda^n_k(f)"'] & \Lambda^n_k(\Y)\end{tikzcd}\]
are pullbacks. One can prove that the converse is true as well.

\begin{coro}An extension $f\colon\X\to \Y$ in $\Simp(\C)$ is central with respect to the reflection of $\Simp(\C)$ into $\Grpd(\C)$ if and only if $f$ is an exact fibration at all dimensions $n\geq 2$ in the sense of Glenn (\cite{G82}).
\end{coro}

\section{Comparison with simplicial sets}
\label{sec:sets}

As noted before, the left adjoint to the nerve functor between groupoids and simplicial sets is the \emph{fundamental groupoid} functor \cite{GZ67}. For a simplicial set $\X$ which satisfies the Kan condition, also called a \emph{quasigroupoid}, this left adjoint can alternatively be described as the \emph{homotopy groupoid} (see \cite{BV73,J02}). One defines the homotopy relation on $X_1$ by saying that two elements (or $1$-simplices) $f,g\in X_1$ are homotopic if and only if there exists $\alpha\in X_2$ such that $d_0(\alpha) = f$, $d_1(\alpha)=g$ and $d_2(\alpha)=s_0d_1f=s_0d_1g$. This is always a reflexive relation (since for a given $f$ one can take $\alpha=s_0f$), and using the Kan condition one can then prove that this is actually an equivalence relation. The homotopy groupoid is then the groupoid whose objects are just the elements of $X_0$, arrows are homotopy classes of $1$-simplices, identities defined by the classes of degenerate $1$-simplices, and composition defined by the existence of fillers for $(2,1)$-horns (with two sided inverses defined by the existence of fillers for the outer horns).

This relation can be interpreted in any regular category as follows : first take the pullback
\begin{equation}\label{diag:pb_homotopy}
\begin{tikzcd}
X_2\times_{X_1} X_0 \pb\ar[d,"\pi_1"']\ar[r,"\pi_2"]& X_0\ar[d,"s_0"] \\ X_2\ar[r,"d_2"']& X_1,
\end{tikzcd}\end{equation}
and then factorize the morphism $(d_0,d_1)\pi_1\colon X_0\times_{X_1} X_2\to X_1\times X_1$ as a regular epimorphism $q\colon P\to R$ followed by a monomorphism $r=(\rho_1,\rho_2)\colon R\to X_1\times X_1$, so that $R$ is a relation on $X_1$. As in the case of sets, this is a reflexive relation; indeed, the simplicial identities imply that
\[(\rho_1,\rho_2)(q\langle d_1,s_0\rangle)=(d_0,d_1)\pi_2\langle d_1,s_0 \rangle = (d_0,d_1)s_0=(1_{X_1},1_{X_1}).\]
This relation coincides with $d_0(D_1\wedge D_2)$ whenever $\X$ satisfies the Kan condition, as we shall now see. In fact it will be helpful to prove a slightly more general result:

\begin{lemma}
	Given any regular epimorphism $f\colon\X\to \Y$ between two simplicial objects, let us consider the pullback
	\[\begin{tikzcd}
	X_1\times_{Y_1} Y_0 \pb \ar[d,"\pi_1"']\ar[r,"\pi_2"]& Y_0\ar[d,"s^\Y_0"] \\ X_1\ar[r,"f_1"']& Y_1.
	\end{tikzcd}\]
	Then $d^\X_0(D^\X_1\wedge F_1)$ is equal to the regular image of $(d_0,d_1)\pi_1 \colon X_1\times_{Y_1} Y_0\to X_0\times X_0$.
	\label{lem:alternative_homotopy}
\end{lemma}

\begin{proof}
	Consider the diagram
	\begin{equation}
		\begin{tikzcd}[column sep={4em,between origins},row sep={3em,between origins}]
		Y_1 \times_{Y_2} X_2 \ar[dr,"{\pi_2'}"']\ar[dd,"{d^\Y_0\times d^\X_0}"', dotted]\ar[rr,"{\pi_1'}"] & & X_2 \ar[dr,"f_2"]\ar[dd,"d^\X_0"', near end] \ar[rr,"{(d_1^\X,d_2^\X)}"]& & X_1 \times X_1 \ar[dd,"{d_0^\X\times d^\X_0}"]\\
		& Y_1 \ar[rr,"s_1^\Y",crossing over,near start]& & Y_2 & \\
		Y_0 \times_{Y_1} X_1 \ar[rr,"{\pi_1}",near end] \ar[dr,"{\pi_2}"'] & & X_1 \ar[dr,"f_1"] \ar[rr,"{(d^\X_0,d^\X_1)}",near start]& & X_0\times X_0 \\
		& Y_0 \ar[from=uu,crossing over,"d^\Y_0"',near start]\ar[rr,"s^\Y_0"'] & & Y_1 \ar[from=uu,"d^\Y_0", crossing over, near start] & \end{tikzcd}\label{eq:big}
	\end{equation}
	where the top and bottom faces of the cube are pullbacks. Since all the vertical solid arrows are split by some degeneracy morphism, and the horizontal morphisms commute with these sections, the dotted arrow is a split epimorphism as well. In particular, the image factorization of $(d^\X_0,d^\X_1)\pi_1$ is the same as that of $(d^\X_0,d^\X_1)\pi_1(d_0\times d_0)=(d^\X_0\times d^\X_0 )(d^\X_1,d^\X_2)\pi_1'$. If we prove that the image of $(d^\X_1,d^\X_2)\pi_2'$ in $X_1\times X_1$ is the equivalence relation $D_1\wedge F_1$, then it would follow that the image of $(d^\X_0\times d^\X_0 )(d^\X_1,d^\X_2)\pi_1'$ is $d_0(D_1\wedge F_1)$, which would conclude the proof.
	
	Since we have a decomposition of $f_2$ given by the diagram
	\[\begin{tikzcd} {X_2} \ar[drr, bend left=15,"{f_2}"]\ar[ddr,bend right=15,"{\lambda_0^2}"'] \ar[dr,"{\theta_0^2}" description] & & \\ & {\Lambda_0^2(\X)\times_{\Lambda_0^2(\Y)} Y_2 } \pb \ar[r, "{\phi_2}"] \ar[d,"{\phi_1}"] & {Y_2} \ar[d,"{\lambda_0^2}"]\\ & {\Lambda_0^2(\X)} \ar[r,"{\Lambda_0^2(f)}"'] & \Lambda_0^2(\Y), \end{tikzcd}\]
	we can rewrite the top pullback in \eqref{eq:big} as the upper rectangle in the following diagram :
	\[\begin{tikzcd}
	X_2\times_{Y_2} Y_1 \pb \ar[r, "q"] \ar[d,"\pi_1'"'] & P\pb \ar[d,"m"] \ar[r]& Y_1\ar[d,"s_1^\Y"] \\
	X_2 \ar[r,"{\theta^2_0}"] \ar[d,"{\lambda_0^2}"']& \Lambda^2_0(\X)\times_{\Lambda^2_0(\Y)} Y_2 \pb \ar[d,"\phi_1"'] \ar[r,"\phi_2"']& Y_2 \ar[d,"{\lambda^2_0}"] \\
	\Lambda^2_0(\X) \ar[r,equal] & \Lambda^2_0(\X) \ar[r,"\Lambda^2_0(f)"'] & \Lambda^2_0(\Y)
	\end{tikzcd}\]
	Now since $\Delta_{Y_1} = (d^\Y_1,d^\Y_2) s^\Y_1 \colon Y_1\to Y_1\times Y_1$, the composition $\lambda^2_0s_1$ is a monomorphism, and thus so is $\phi_1m$. Since $\lambda^2_0$ is the regular epimorphism in the factorization of $(d^\X_1,d^\X_2)\colon X_2\to X_1\times X_1$, $P$ is the image of the morphism $(d^\X_1,d^\X_2)\pi_1'$. On the other hand, the right-hand rectangle above coincides with the left-hand square in the rectangle
	\[\begin{tikzcd}
	P \pb \ar[d,"\phi_1m"'] \ar[r] & Y_1\pb \ar[d] \ar[r,equal] & Y_1 \ar[d,"{\Delta}"] \\
	\Lambda^2_0(\X) \ar[r,"{\Lambda^2_0(f)}"'] & \Lambda^2_0(\Y) \ar[r] & Y_1\times Y_1.
	\end{tikzcd}\]
	Since the two squares are pullbacks, the whole rectangle is one as well. But this is the same as the outer rectangle in
	\[\begin{tikzcd}
	P \pb \ar[d,"\phi_1m"'] \ar[r] & F_1\pb \ar[d] \ar[r] & Y_1 \ar[d,"{\Delta}"] \\
	\Lambda^2_0(\X) \ar[r,"{}"] & X_1\times X_1 \ar[r,"{f_1\times f_1}"'] & Y_1\times Y_1,
	\end{tikzcd}\]
	where the two squares are again pullbacks. Thus $P$ coincides with the intersection $D_1\wedge F_1$, which concludes the proof.
\end{proof}

As a consequence we have

\begin{prop}
	If $\X$ is a Kan complex in a regular category $\C$, the relation $H_1(\X)$ coincides with the image of the morphism $(d^\X_0,d^\X_1)\pi_1\colon X_0\times_{X_1} X_2\to X_1\times X_1$, where $\pi_1$ is determined by the pullback \eqref{diag:pb_homotopy}.
\end{prop}

\begin{proof}
	It suffices to apply \autoref{lem:alternative_homotopy} to the case where $f_1=\epsilon_\X\colon Dec(\X)\to \X$.
\end{proof}

\begin{remark}If one sees a Kan complex as a quasigroupoid or $\infty$-groupoid, then the left adjoint to the nerve or inclusion functor $\Grpd\to \mathbf{Kan}$ is in a sense a "strictification", which turns quasigroupoids into actual groupoids.
\end{remark}

The equivalence relation $d_0(D_1\wedge D_2\wedge F_2)$ which appears in our characterization of central extensions admits an alternative construction, similar to that of $H_1(\X)$.

More precisely, if we take now $L$ to be the limit of the lower part of the diagram
\begin{equation}\label{eq:limit}\begin{tikzcd} & & L\ar[dll,"\rho_1"',dotted]\ar[d,"\rho_2",dotted]\ar[drr,"\rho_3",dotted]& & \\ X_0 \ar[dr,"s_0"'] & & X_2\ar[dl,"d_2"] \ar[dr,"f_2"'] & & Y_1 \ar[dl,"s_0"] \\ & X_1 & & Y_2 &\end{tikzcd}\end{equation}
(where the dotted arrows form the limit cone), then we have

\begin{prop}
	If $\X,\Y$ are Kan complexes and $f\colon \X\to \Y$ is a Kan fibrations in a regular category $\C$, the relation $d_1(F_2\wedge D_0\wedge D_2)$ coincides with the image of the morphism $(d^\X_0,d^\X_1)\rho_2\colon X_0\times_{X_1} X_2\to X_1\times X_1$.
\end{prop}

\begin{proof}
	First, note that the limit in diagram \eqref{eq:limit} can also be obtained as the pullback
	\[\begin{tikzcd} L\pb  \ar[r,"{(\rho_1,\rho_3)}"] \ar[d,"\rho_2"'] & X_0\times Y_1 \ar[d,"{s_0\times s_0}"] \\ X_2 \ar[r, "{(d_2,f_2)}"']& X_1\times Y_2.\end{tikzcd}\]
	Now the image of the morphism $(d_2,f_2)$ is the pullback $X_1\times_{Y_1}Y_2$ of $f_1$ along $d_2$. Moreover, we have
	\begin{align*}f_0\rho_1 & =d_0s_0f_0\rho_1=d_0f_1s_0\rho_1=d_0f_1d_2\rho_2 = d_0d_2f_2\rho_2 \\ & =d_0d_2s_0\rho_3=d_1d_0s_0\rho_3=d_1\rho_3,\end{align*}
	so that $(\rho_1,\rho_3)$ factors through $X_0\times_{Y_0} X_1$. Thus the pullback square above factorises as a rectangle
	\[\begin{tikzcd} L \ar[r,"{\langle\rho_1,\rho_3\rangle}"] \ar[d,"\rho_2"'] & {X_0\times_{Y_0} X_1} \ar[d,"{s_0\times_{s_0}s_0}"] \ar[r]& {X_0\times X_1} \ar[d,"{s_0\times s_0}"] \\ X_2 \ar[r, "{\langle d_2,f_2\rangle}"]& {X_1\times_{Y_1} Y_2} \ar[r]& {X_1\times Y_2},\end{tikzcd}\]
	and one can easily show that the right-hand square is a pullback, and as a consequence so is the left-hand side square. But this square is exactly the pullback that appears if we apply \autoref{lem:alternative_homotopy} to the induced morphism $\langle \epsilon_\X,Dec(f)\rangle \colon Dec(\X)\to \X\times_\Y Dec(\Y)$, which is a regular epimorphism between simplicial objects because the square
	\[\begin{tikzcd} Dec(\X) \ar[r,"Dec(f)"]\ar[d,"\epsilon_\X"'] & Dec(\Y) \ar[d,"\epsilon_\Y"]\\ \X \ar[r,"f"']& \Y \end{tikzcd}\]
	is a double extension in $\C$ for all $n$.
\end{proof}

\section{The relative monotone-light factorization system}

In this section we assume

In order to prove that our Galois structure admits a relative monotone-light factorization system, we use the following criterion, due to Carboni, Janelidze, Kelly and Par\'e in the absolute case and to Chikhladze in the relative case :

\begin{prop}[\cite{CJKP97,C04}]Let $(\C,\mathcal{X},I,\F)$ be an admissible Galois structure. The class $\F$ admits monotone-light factorization if for each object $B$ of $\C$ there is an effective $\F$-descent morphism $p\colon C\to B$ where $C$ is a stabilizing object, i.e. an object such that if $h=me$ is the $(\E,\M)$-factorization of any morphism $h\colon X\to C$, then any pullback of $e$ along a morphism in $\F$ is still in $\E$.
\end{prop}

We will prove that, in our case, the shifting $Dec(\X)$ of a simplicial object $\X$ is always stabilizing. For this it suffices to prove that exact objects are stabilizing since we have :

\begin{prop}[\cite{EGoV12}, Proposition 3.9]Any simplicial object that is contractible and also satisfies the Kan condition is exact.
\end{prop}

As a consequence, if $\X$ satisfies the Kan condition, then its shifting $Dec(\X)$ is exact.

We will need the following characterization of images in regular categories:
\begin{prop}[\cite{CKP93}]Let $f\colon X\to Z$ and $g\colon Y\to Z$ be two morphisms in a regular category $\C$. Then $g$ factors through the regular image of $f$ if and only if there exist an object $W$ of $\C$ with a morphism $h\colon W\to X$ and a regular epimorphism $q\colon W\to Y$ such that $fh=gq$.\label{prop:image}
\end{prop}

\begin{lemma}\label{lem:inequality_exact}
	If $\Y$ is exact at $Y_2$, and $f\colon \X\to \Y$ is a regular epimorphism in $\Simp(\C)$, then
	\[d_0(D_1\wedge D_2)\wedge F_1 = d_0(D_1\wedge D_2\wedge F_2).\]
\end{lemma}

\begin{proof}
	The inequality
	\[d_0(D_1\wedge D_2\wedge F_2)\leq d_0(D_1\wedge D_2)\wedge F_1\]
	always holds. To prove the converse, we consider the monomorphism $\phi=(\phi_1,\phi_2)$ into $X_1\times X_1$ corresponding to the equivalence relation $d_0(D_1\wedge D_2)\wedge F_1$. This relation is smaller than $d_0(D_1\wedge D_2)$, so that, by the characterization given in \autoref{prop:image} and the alternative construction of $d_0(D_1\wedge D_2)$ given in \autoref{sec:sets}, there must exist a regular epimorphism $p\colon Z\to d_0(D_1\wedge D_2)\wedge F_1$ and a morphism $\alpha=\langle\alpha_1,\alpha_2\rangle \colon Z\to X_2\times_{X_1}X_0$ such that $d_0\alpha_1 = \phi_1p$ and $d_1\alpha_1=\phi_2p$. Since, moreover, it is smaller than $F_1$, we have $f_1d_0\alpha_1=f_1d_1\alpha_1$, which can be rewritten $d_0f_2\alpha_1=d_1f_2\alpha_1$.
	
	Now consider the morphisms
	\begin{align*}
	y_0=y_1 & =s_1d_0f_2\alpha_1=s_1d_1f_2\alpha_1 \\
	y_2 & =s_0d_1f_2\alpha_1\\
	y_3 & =f_2\alpha_1.
	\end{align*}
	
	One can check that the identity $d_iy_j=d_{j-1}y_i$ holds for all $0\leq i<j\leq 3$, so that these morphisms determine a morphism $y$ from $Z$ to the third simplicial kernel $K_3(\Y)$, and we can consider the pullback
	\[\begin{tikzcd}
	Z' \pb \ar[d,"\alpha'"'] \ar[r,"p'"] & Z\ar[d,"y"]\\ Y_3 \ar[r,"\kappa_3"'] & K_3(\Y).
	\end{tikzcd}\]
	$\Y$ being exact at $Y_2$ means that $\kappa_3$ is a regular epimorphism, and, as a consequence, so is $p'$. Consider now the morphisms
	\[x_0=s_1d_0\alpha_1p'\]
	\[x_1=s_1d_1\alpha_1p'\]
	\[x_3=\alpha_1p',\]
	from $Z'$ to $X_2$. One can check that the identity $d_ix_j=d_{j-1}x_i$ holds for all $i<j$ and $i\neq 2\neq j$, thus they determine a morphism $x\colon Z'\to \Lambda^3_2(\X)$; and, moreover, we have
	\[d_i\alpha'=\mu_i\kappa_3 \alpha'=\mu_i yp'=y_ip'=f_2x_i\]
	for $i=0,1,3$, which implies that $\lambda^3_2\alpha'=\Lambda^3_2(f)x$. Thus $x$ and $\alpha'$ induce a morphism $Z'\to \Lambda^3_2(\X) \times_{\Lambda_2^3(\Y)} Y_3$. Consider then the pullback
	\[\begin{tikzcd} Z''\ar[d,"\alpha''"'] \ar[r,"p''"] & Z'\ar[d,"{\langle x,\alpha'\rangle}"]\\ X_3 \ar[r,"\theta^3_2"'] & \Lambda^3_2(\X)\times_{\Lambda^3_2(\Y)} Y_3.
	\end{tikzcd}\]
	Since $\theta^3_2$ is a regular epimorphism, so is $p''$, and by construction we have $d_i\alpha''=x_ip''$ for $i=0,1,3$ and $f_3\alpha''=\alpha' p''$. Now the morphism $d_2\alpha''\colon Z''\to X_2$ is such that
	\[f_2d_2\alpha''=d_2f_3\alpha''=d_2\alpha' p''=y_2p'p''=s_0d_1f_2\alpha_1 p' p''\]
	and
	\[d_2d_2\alpha''=d_2d_3\alpha''= d_2x_3p''=d_2\alpha_1 p' p''=s_0\alpha_2 p' p'';\]
	hence there exists a unique morphism $\beta\colon Z''\to L$ (where $L$ is defined by \eqref{eq:limit}) such that $\rho_1\beta= \alpha_2 p' p''$, $\rho_2\beta=d_2\alpha''$ and $\rho_3\beta=f_1d_1\alpha_1 p'p''$. Now we can check that
	\begin{align*} (d_0,d_1)\rho_2\beta & =(d_0,d_1)d_2\alpha''=(d_1d_0, d_1d_1)\alpha''=(d_1x_0,d_1x_1)p'' \\ & =(d_1 s_1d_0, d_1 s_1d_1 )\alpha_1p'p'' =(d_0,d_1)\alpha_1 p'p''\\ & = (\phi_1,\phi_2)pp'p''.\end{align*}
	This proves that $d_0(D_1\wedge D_2)\wedge F_1\leq d_0(D_1\wedge D_2\wedge F_2)$.
\end{proof}

\begin{lemma}
	If $\Y$ is exact, then $\Y$ is stabilizing : given any morphism $f\colon \X\to \Y$, the induced morphism $\langle f,\eta_\X \rangle \colon \X\to \Y\times_{\Pi_1(\Y)} \Pi_1(\X)$ is stably in $\E$.
\end{lemma}

\begin{proof}
	To simplify the diagrams, we denote $\bP= \Y\times_{\Pi_1(\Y)} \Pi_1(\X)$. Let us consider a pullback square
	\begin{equation}\label{eq:pb2}\begin{tikzcd}
	\Q \pb\ar[d,"h"'] \ar[r,"g'"] & \X \ar[d,"{\langle f,\eta_\X \rangle}"] \\ \Z \ar[r,"g"'] & \bP \end{tikzcd}\end{equation}
	with $g$ a regular epimorphism in $\Simp(\C)$.
	
	We need to prove that $\Pi_1(h)\colon \Pi_1(\Q)\to \Pi_1(\Z)$ is invertible. Since it is a morphism between internal groupoids, it is enough to prove that $\Pi_1(h)_0$ and $\Pi_1(h)_1$ are invertible. Note that the functor $\Pi_1$ leaves the objects $X_0$ unchanged, and thus $\langle f_0,\eta_0\rangle$ is an isomorphism, and thus so are $h_0$ and $\Pi_1(h)_0$. So we only need to prove is that $\Pi_1(h)_1$ is an isomorphism.
	
	Since $\Grpd(\C)$ is a Birkhoff subcategory of $\Simp(\C)$ and $h$ is a regular epimorphism, the square
	\[\begin{tikzcd} \Q \ar[d,"h"']\ar[r,"\eta_\Q"] & \Pi_1(\Q) \ar[d,"\Pi_1(h)"] \\ \Z \ar[r,"\eta_\Q"']& \Pi_1(\Z)
	\end{tikzcd}\]
	is a double extension in $\Simp(\C)$, and thus the square
	\begin{equation}\begin{tikzcd}
	Q_1 \ar[d,"h_1"']\ar[r,"(\eta_\Q)_1"] & \frac{Q_1}{H_1(\Q)} \ar[d,"\overline{h_1}"] \\ Z_1 \ar[r,"(\eta_\Z)_1"']& \frac{Z_1}{H_1(\Z)}
	\end{tikzcd}\label{eq:pushout_square}\end{equation}
	is a (regular) pushout in $\C$. This already proves that $\overline{h_1}=\Pi_1(h)_1$ is a regular epi. Now if there exists a morphism $t\colon Z_1\to Q_1/H_1(\Q)$ such that $t h_1=(\eta_\Q)_1$, then using the universal property of the pushout \eqref{eq:pushout_square} we can construct a retraction for $\overline{h_1}$, which proves that it is an isomorphism. So we are left to prove that such a morphism $t$ exists; since $h_1$ is a regular epimorphism, it is enough to prove that $Eq[h_1]\leq H_1(\Q)$.
	
	To prove this, we denote $\psi_1,\psi_2\colon Eq[h_1]\to Q_1$ the two projections of the kernel pair. Then the commutativity of \eqref{eq:pb2} (or rather, the corresponding commutative square involving $h_1$ in $\C$) implies that
	\[g_1'(Eq[h_1])\leq Eq[\langle f_1,(\eta_\X)_1] = F_1 \wedge H_1(\X)=d_0(D_1\wedge D_2\wedge F_2)\]
	where the last equality is given by \autoref{lem:inequality_exact}. As a consequence, we know that there must exist a morphism $\alpha\colon A\to L$ and a regular epimorphism $p\colon A\to Eq[h_1]$ such that $(d_0,d_1)\rho_2\alpha=(g_1'\times g_1') (\psi_1,\psi_2)p$.
	
	We now prove that $\langle f_2,\eta_2\rangle \rho_2\alpha$ factors through a degeneracy of $\bP$. More precisely, we prove that
	\begin{equation}\label{eq:factor_deg}
	\langle f_2,\eta_2\rangle \rho_2\alpha=s^{\bP}_0d_0^{\bP}\langle f_2,\eta_2\rangle \rho_2\alpha.\end{equation}
	Since the degeneracy morphism $s_0^{\bP}$ is induced by those of $\Pi_1(\X)$ and $\Y$, it is enough to prove that $f_2\rho_2\alpha$ and $\eta_2\rho_2\alpha$ factorize in the same manner through $s_0^{\Y}$ and $s_0^{\Pi_1(\X)}$ respectively.
	
	By construction we must have
	\[s^{\Y}_0d_{0}^{\Y}f_2\rho_2\alpha= s^{\Y}_0d_{0}^{\Y}s^{\Y}_0\rho_3\alpha=s^{\Y}_0\rho_3\alpha=f_2\rho_2\alpha.\]
	On the other hand we have
	\[d^{\Pi_1(\X)}_0\eta_2 \rho_2\alpha=d^{\Pi_1(\X)}_0s_0^{\Pi_1(\X)}d^{\Pi_1(\X)}_0\eta_2\rho_2\alpha\]
	by the simplicial identities, and
	\begin{align*}d^{\Pi_1(\X)}_1\eta_2 \rho_2\alpha & =\eta_1d_1^{\X}\rho_2\alpha=\eta_1g_1'\psi_2p=g_1h_1\psi_2p=g_1h_1\psi_1p \\ & = \eta_1g_1'\psi_1p=\eta_1d_0^{\X}\rho_2\alpha=d^{\Pi_1(\X)}_0\eta_2 \rho_2\alpha \\ & = d^{\Pi_1(\X)}_1 s_0^{\Pi_1(\X)}d^{\Pi_1(\X)}_0\eta_2\rho_2\alpha.\end{align*}
	
	By construction, the two morphisms $d^{\Pi_1(\X)}_0,d^{\Pi_1(\X)}_1\colon \frac{X_2}{H_2(\X)} \to \frac{X_1}{H_1(\X)}$ are jointly monic, and thus these equalities imply that
	\[\eta_2 \rho_2\alpha=s_0^{\Pi_1(\X)}d^{\Pi_1(\X)}_0\eta_1\rho_2\alpha,\]
	and this in turn implies that \eqref{eq:factor_deg} hold. From this we find that
	\begin{align*}\langle f_2,\eta_2\rangle \rho_2\alpha & =s^{\bP}_0d_0^{\bP}\langle f_2,\eta_2\rangle \rho_2\alpha \\ & = s^{\bP}_0\langle f_1,\eta_1 \rangle d_0^{\X}\rho_2\alpha \\ & = s^{\bP}_0\langle f_1,\eta_1 \rangle g'_1 \psi_1 p \\ & = s^{\bP}_0 g_1h_1 \psi_1 p \\ & = g_2s^{\Z}_0 h_1 \psi_1 p. \end{align*}
	
	Since $Q_2$ is the pullback of $g_2$ along $\langle f_2,\eta_2\rangle$, there is a unique morphism $\alpha'\colon A\to Q_2$ such that $h_2\alpha' =s^{\Z}_0 h_1 \psi_1 p$ and $g_2'\alpha'= \rho_2\alpha$. From this, we find that
	\[g_1'd^{\Q}_0\alpha' = d^{\X}_0 g_2'\alpha'= d^{\X}_0\rho_2 \alpha = g_1'\psi_1 p \]
	and
	\[h_1d^{\Q}_0\alpha' = d^{\Z}_0 h_2\alpha' = d^{\Z}_0 s^{\Z}_0 h_1 \psi_1 p=h_1\psi_1 p,\]
	and since $g_1'$ and $h_1$ are jointly monic, we have $d^{\Q}_0 \alpha'=\psi_1 p$, and similarly $d^{\Q}_1 \alpha'=\psi_2 p$.
	
	Now we prove that $d_2^{\Q}\alpha'=s_0^{\Q}d^{\Q}_0d^{\Q}_2\alpha'$; from the definition of $Q_1$ it suffices to check that the identity holds after composing both sides with each of the morphisms $h_1$ and $g_1'$. We have
	\begin{align*}
	h_1s_0^{\Q}d^{\Q}_0d^{\Q}_2\alpha' & = d^{\Q}_2h_2s_0^{\Q}d^{\Q}_0\alpha' = d^{\Z}_2s_0^{\Z}d^{\Z}_0h_2\alpha' = s_0^{\Z}d^{\Z}_1d^{\Z}_0 s^{\Z}_0 h_1 \psi_1 p \\
	&  =  s_0^{\Z} d^{\Z}_1 h_1 \psi_1 p =d^{\Z}_2 s^{\Z}_0 h_1\psi_p = d_2^{\Z}h_2\alpha' = h_1 \alpha'\end{align*}
	and
	\begin{align*}
	g_1's_0^{\Q}d^{\Q}_0d^{\Q}_2\alpha' & = s_0^{\X}d^{\X}_0d^{\X}_2 g_2'\alpha' = s_0^{\X}d^{\X}_0d^{\X}_2 \rho_2\alpha = s_0^{\X}d^{\X}_0s^{\X}_0 \rho_1\alpha \\
	& = s_0^{\X} \rho_1\alpha = d^{\X}_2 \rho_2\alpha = d^{\X}_2 g_2'\alpha' = g_1' d^{\Q}_2 \alpha'
	\end{align*}
	
	Thus $\alpha'$ factorizes through the pullback of $s^{Q}_0$ along $d^{Q}_2$, and thus $(\psi_1,\psi_2)p=(d^{\Q}_0,d^{\Q}_1)\alpha'$ factorizes through the inclusion of $H_1(\Q)$ in $Q_1\times Q_1$, which concludes the proof.
\end{proof}

As a consequence, we then have
\begin{thm}
	If $\C$ is an exact Mal'tsev category, the Galois structure induced by the reflection of simplicial objects into internal groupoids admits a relative monotone-light factorization system for regular epimorphisms $(\E',\M^*)$, where $\E'$ is the class of morphisms stably inverted by $\Pi_1$ and $\M^*$ is the class of central extensions of this Galois structure.
\end{thm}

\section{Truncated simplicial objects and weighted commutators}

For all $n\geq 2$, we can define a nerve functor $\Grpd(\C)\hookrightarrow \Simp_n(\C)$; this amounts to compose the usual nerver functor $\Grpd(\C)\to \Simp(\C)$ with the truncation functor $\Simp(\C)\to\Simp_n(\C)$. The characterization of groupoids in truncated simplicial objects is then identical.

Moreover, the construction of the equivalence relations $H_n(\X)$ does not depend on the objects $X_m$ for $m>n$. Thus $\Grpd(\C)$ can also be seen as a Birkhoff subcategory of $\Simp_n(\C)$, with the reflection defined in the same way, in the sense that the reflectors commute with the truncation functor. The characterization of central extensions also extends in the same way. Note that for $n=2$, truncated simplicial sets coincide with internal \emph{precategories} in the sense of \cite{J91a}.

The forgetful functor $\Grpd(\C)\hookrightarrow \Simp_1(\C)=\RG(\C)$ also coincides with the composition of the nerve functor with the truncation functor $\Simp(\C)\to \RG(\C)$, and it is also fully faithful \cite{CPP92}. On the other hand, this time the reflection does not commute with the truncation, as the construction of $H_1(\X)$ depends on $X_2$ and the face morphisms $X_2\to X_1$. In fact, the reflection $\RG(\C)\to \Grpd(\C)$ is obtained by taking the quotient of $X_1$ by the Smith-Pedicchio commutator $[D_0,D_1]_{SP}$ (\cite{P95}). The central extensions of reflexive graphs in exact Mal'tsev categories (with coequalizers) with respect to this adjunction have been characterized in \cite{DG18}. Note that this commutator is preserved by regular images, and is always smaller than the intersection; as a consequence, we always have the inequalities
\begin{equation}\label{eq:comm_ineq}[D_0,D_1]_{SP}\leq H_1(\X)=d_2(D_0\wedge D_1) \leq D_0\wedge D_1.\end{equation}
It turns out that this reflection can also be obtained by applying our results, at least when the category $\C$ is finitely cocomplete. 

Indeed, in that case the truncation functor $\Simp(\C)\to \RG(\C)$ is right adjoint to the $1$-skeleton functor $Sk_1$ \cite{D75}, which can be defined by taking left Kan extensions along the inclusion $\Delta_2^{op}\to \Delta^{op}$. Now since the inclusion $\Grpd(\C)\to \RG(\C)$ is the composition of the nerve functor and the truncation $Tr_1$, the functor $\Pi_1 Sk_1$ must be a left adjoint to this inclusion. Thus our results can be used to give an alternative description of the Smith-Pedicchio commutator as the equivalence relation $H_1(Sk_1(X_1,X_0,d_0,d_1,s_0))$.

Let us make this construction explicit. The object $X_2=(Sk_1(X_1,X_0,d_0,d_1,s_0))_2$ is the pushout $X_1+_{X_0}X_1$ of $s_0\colon X_0\to X_1$ along itself, with $s_0$ and $s_1$ the two canonical morphisms $X_1\to X_1+_{X_0}X_1$. In order to satisfy the simplicial identities we must then define $d_0$ to be the unique morphism for which $d_0s_0=1$ and $d_0s_1=s_0d_0$, which we denote $[1,s_0d_0]\colon X_1+_{X_0}X_1\to X_1$; similarly, we must have $d_1=[1,1]$ and $d_2=[s_0d_1,1]$.

In the case where $\C$ is not only exact Mal'tsev but also semi-abelian (\cite{JMT02,BB04}), there is, for every object $X$ of $\C$, an order-preserving bijection between equivalence relations on $X$ and normal subobjects of $X$, which is also compatible with regular images. Accordingly, our results can be easily translated in terms of normal subobjects, by replacing every kernel pair by the kernel of the corresponding morphism.

In the case where $X_0=0$ is the zero object in $\C$, $X_2$ is simply the coproduct $X_1+X_1$, and the face morphisms are just the morphisms $[1,0],[1,1],[0,1]$. Then our construction of $d_1(D_0\wedge D_2)$ is nothing but the Higgins commutator $[X_1,X_1]_H$ (which coincides with the Smith-Pediccio commutator $[\nabla_{X_1},\nabla_{X_1}]_{SP}$), as defined in \cite{H56,MM10}. In general, $d_1(D_0\wedge D_2)$ coincides with a weighted commutator (\cite{GJU14}) :
\begin{thm}
When $\C$ is a semi-abelian category, the subobject $d_1(Ker(d_0)\wedge Ker(d_2))$ coincides with the weighted commutator $[Ker(d_0),Ker(d_1)]_{s_0\colon X_0\to X_1}$.
\end{thm}

\begin{proof}Let us denote $K_i\leq X_1$ the kernel of $d_i\colon X_1\to X_0$ (for $i=0,1$). We recall from \cite{GJU14} the construction of the weighted commutator $[K_0,K_1]_{X_0}$ : we first define the morphism $\psi$ as the morphism making the diagram
\[\begin{tikzcd} X_0+K_0+K_1 \ar[drr,bend left=15,"{[\iota_1,0,\iota_2]}"] \ar[ddr,bend right=15,"{[\iota_1,\iota_2,0]}"'] \ar[dr,"\psi" description]& & \\ & (X_0+K_0) \times_{X_0}(X_0+K_1) \pb \ar[r] \ar[d] & X_0+K_1\ar[d,"{[1,0]}"] \\ & X_0+K_0 \ar[r,"{[1,0]}"'] & X_0
\end{tikzcd}\]
commute. Then $[K_0,K_1]_{X_0}$ is the image of the kernel of $\psi$ under the morphism $[s_0,k_0,k_1]\colon X_0+K_0+K_1\to X_1$.

To prove the equivalence, consider the following commutative diagram :
\[\begin{tikzcd} X_0+K_0 \ar[d,"{[s_0,k_0]}"'] \ar[r,shift right,"{[1,0]}"'] & X_0 \ar[r,shift left,"\iota_1"]\ar[l, shift right,"\iota_1"'] \ar[d,equal] & X_0+K_1 \ar[l,shift left,"{[1,0]}"] \ar[d,"{[s_0,k_1]}"]\\
X_1 \ar[r,shift right,"{d_0}"'] & X_0 \ar[r,shift left,"s_0"]\ar[l, shift right,"s_0"'] & X_1. \ar[l,shift left,"{d_1}"]
\end{tikzcd}\]
Since all the vertical morphisms are regular epimorphisms, the induced morphism between the pushouts of the upper and lower spans (i.e. the cokernel pairs of $\iota_1$ and $s_0$), which we will denote by $\gamma$, is also a regular epimorphism. This gives a commutative cube

\[\begin{tikzcd}[column sep={4.4em,between origins},row sep={3.3em,between origins}]
X_0+K_0+K_1 \ar[dr,"{[\iota_1,\iota_2,0]}" pos=0.6]\ar[dd,"\gamma"']\ar[rr,"{[\iota_1,0,\iota_2]}"] & & X_0+K_1 \ar[dr,"{[1,0]}"]\ar[dd,"{[s_0,k_1]}"', near end] & \\
& X_0+K_0 \ar[rr,"{[1,0]}" pos=0.6,crossing over]& & X_0\ar[dd,equal]\\
X_1+_{X_0}X_1 \ar[rr,"{[1,s_0d_0]}"' pos=0.67] \ar[dr,"{[s_0d_1,1]}"'] & & X_1\ar[dr,"d_1"] & \\
& X_1\ar[from=uu,crossing over,"{[s_0,k_0]}"'pos=0.2]\ar[rr,"d_0"'] & & X_0\end{tikzcd}\]
where every edge is a regular epimorphism. In fact this cube is a triple extension, as it can be seen as a split epimorphism between (vertical) double extensions. As a consequence the induced square
\[\begin{tikzcd} X_0+K_0+K_1 \ar[r,"\psi"]\ar[d,"\gamma"'] & (X_0+K_0) \times_{X_0}(X_0+K_1) \ar[d] \\
X_1+_{X_0}X_1 \ar[r,"{\langle d_0,d_2\rangle}"'] & X_1 \times_{X_0} X_1 \end{tikzcd}\]
is a double extension; in particular, we have
\[\gamma(Ker(\psi))=Ker(\langle d_0,d_2\rangle )=Ker(d_0)\wedge Ker(d_2).\] Now we also have
\[d_1\gamma=[1,1] ([s_0,k_0]+[s_0,k_1])=[s_0,k_0,k_1],\]
and thus the image of $Ker(\psi)$ under $[s_0,k_0,k_1]$ is $d_1(Ker(d_0)\wedge Ker(d_2))$, which completes the proof.
\end{proof}

\begin{coro}For any reflexive graph in a semi-abelian category $\C$, the weighted commutator $[Ker(d_0),Ker(d_1)]_{s_0\colon X_0\to X_1}$ of the kernels of $d_0$ and $d_1$ coincides with their Ursini commutator $[Ker(d_0),Ker(d_1)]_{Urs}$ as defined by Mantovani in \cite{M12}.
\end{coro}

\begin{proof}
This just follows from the fact that the Ursini commutator is the normalization of the Smith-Pedicchio commutator of the corresponding equivalence relations.
\end{proof}

We have shown that using the left adjoint of the truncation functor produced a simplicial object for which the first inequality of \eqref{eq:comm_ineq} is an equality, so that $H_1(\X)$ is as small as possible. We can also define a right adjoint $R$ to the truncation functor $T$, using right Kan extensions along the inclusion $\Delta_1^{op}\to \Delta^{op}$. Such a right extension amounts to iteratively define $X_n$ as the simplicial kernel of the truncated simplicial object $\begin{tikzcd} X_{n-1}\ar[r, shift left,"d_0"] \ar[r,shift right,"d_{n-1}"'] & X_{n-2}\dots \end{tikzcd}$, and the face morphisms $d_i\colon X_n\to X_{n-1}$ as the canonical projections. If we apply this construction, then the induced equivalence relation $H_1(\X)$ turns out to be equal to $D_0\wedge D_1$, so that this time $H_1(\X)$ is as big as possible. In fact we can prove something a bit more general :

\begin{prop}If $\X$ is a simplicial object exact at $X_1$, i.e. if $\kappa_2\colon X_2\to K_2(\X)$ is a regular epimorphism, then $d_0(D_1\wedge D_2)=D_0\wedge D_1$.
\end{prop}

\begin{proof}
Consider the following diagram, where all the squares are pullbacks :
\[\begin{tikzcd} X_2\times_{X_1} X_0 \pb \ar[d,"\pi_1"'] \ar[r,"q"] & K_2(\X)\times_{X_1} X_0 \pb \ar[d]\ar[r] & X_0\ar[d,"{s_0}"] \\ 
X_2 \ar[dr,"{\langle d_0,d_1\rangle}"']]\ar[r,"{\kappa_2}"'] & K_2(\X) \ar[d,"{(\nu_0,\nu_1)}"'] \pb \ar[r, "{\nu_2}"]& X_1 \ar[d,"{(d_0,d_1)}"]\\
& D_0 \ar[r,"{d_1\times d_1}"'] & X_0\times X_0. \end{tikzcd}\]
By definition, $\nu_2\kappa_2=d_2$, and thus the upper rectangle is the pullback of $d_2$ along $s_0$, i.e. it is the same pullback as in \eqref{diag:pb_homotopy}; thus $d_0(D_1\wedge D_2)$ is the image of the composition $\langle d_0,d_1\rangle \pi_1$. Since the upper left square is a pullback and $\kappa_2$ is a regular epimorphism by hypothesis, $q$ is a regular epimorphism, and thus the image of this morphism $\langle d_0,d_1\rangle \pi_1$ is the image of the middle vertical morphism. Moreover, the right-hand rectangle is the pullback of $d_1\times d_1$ along $\Delta_{X_0}$, and thus this middle vertical morphism is a monomorphism, and the corresponding subobject of $D_0$ coincides with $D_0\wedge D_1$, which concludes the proof.\end{proof}

\section*{Acknowledgements}
This work was part of my PhD, funded by a Research Fellowship of the FNRS. I would like to thank my PhD advisor Marino Gran for his suggestions and helpful guidance. I also thank the anonymous referee for their helpful comments.

\bibliography{biblio}
\bibliographystyle{abbrv}

\end{document}